\documentclass[11pt]{amsart}
\usepackage{amsthm,mathrsfs,amssymb}
\usepackage{amsmath,verbatim}
\usepackage[active]{srcltx}
\usepackage[matrix,arrow,curve]{xy}
\usepackage{verbatim}
\usepackage{setspace}
\usepackage{amsfonts}
\usepackage{bbold}
\usepackage{bbm,dsfont}
\usepackage[margin=1in]{geometry}
\usepackage[foot]{amsaddr}

\newcounter{satz}
\newtheorem{thm}[satz]{Theorem}
\newtheorem{prop}[satz]{Proposition}
\newtheorem{lemma}[satz]{Lemma}
\newtheorem{cor}[satz]{Corollary}

\theoremstyle{definition}

\newtheorem{defn}[satz]{Definition}

\newcommand{\SK}{\operatorname{Skew}}

\newcommand{\RM}{R_\calM}
\newcommand{\DR}{\mathrm{DR}}
\newcommand{\Fil}{\mathrm{Fil}}

\numberwithin{equation}{section}

\usepackage{verbatim}

\newcommand{\scrA}{\mathscr{A}}
\newcommand{\scrB}{\mathscr{B}}

\newcommand{\scrL}{\mathscr{L}}

\newcommand{\scrP}{\mathscr{P}}

\newcommand{\calA}{\mathcal{A}}

\newcommand{\calC}{\mathcal{C}}

\newcommand{\calH}{\mathcal{H}}

\newcommand{\calL}{\mathcal{L}}
\newcommand{\calM}{\mathcal{M}}

\newcommand{\calO}{\mathcal{O}}

\newcommand{\calS}{\mathcal{S}}
\newcommand{\calT}{\mathcal{T}}

\newcommand{\calX}{\mathcal{X}}
\newcommand{\calY}{\mathcal{Y}}

\newcommand{\fraka}{\mathfrak{a}}
\newcommand{\frakb}{\mathfrak{b}}
\newcommand{\frakc}{\mathfrak{c}}
\newcommand{\frakd}{\mathfrak{d}}

\newcommand{\frakr}{\mathfrak{r}}

\newcommand{\frakC}{\mathfrak{C}}

\newcommand{\CC}{\mathbbm{C}}
\newcommand{\DD}{\mathbbm{D}}

\newcommand{\FF}{\mathbbm{F}}

\newcommand{\HH}{\mathbbm{H}}

\newcommand{\QQ}{\mathbbm{Q}}
\newcommand{\RR}{\mathbbm{R}}

\newcommand{\ZZ}{\mathbbm{Z}}

\renewcommand{\ker}{\operatorname{ker}}

\newcommand{\isomto}{\xrightarrow{\raisebox{-5pt}{$\sim$}}}

\newcommand{\rar}{\rightarrow}

\newcommand{\lra}{\longrightarrow}

\newcommand{\hra}{\hookrightarrow}	
\newcommand{\thrar}{\twoheadrightarrow}

\newcommand{\wt}{\widetilde}
\newcommand{\wh}{\widehat}
\newcommand{\ul}{\underline}
\newcommand{\ol}{\overline}

\newcommand{\Ab}{\mathrm{Ab}}
\newcommand{\End}{\mathrm{End}}
\newcommand{\Hom}{\mathrm{Hom}}

\newcommand{\Lie}{\mathrm{Lie}}
\newcommand{\Spec}{\operatorname{Spec}}

\newcommand{\Gal}{\mathrm{Gal}}
\newcommand{\Qbar}{\overline{\QQ}}

\newcommand{\Tr}{\operatorname{Tr}}

\newcommand{\Res}{\mathrm{Res}}
\newcommand{\Aut}{\mathrm{Aut}}
\newcommand{\id}{\mathrm{id}}

\newcommand{\Corr}{\mathrm{Corr}}

\newcommand{\Ob}{\operatorname{Ob}}

\newcommand{\Herm}{\mathrm{Herm}}
\newcommand{\Pic}{\mathrm{Pic}}

\newcommand{\Nm}{\mathrm{Nm}}

\newcommand{\sh}{\mathrm{\makebox[2pt]{}sh}}

\renewcommand{\Im}{\operatorname{Im}}
\renewcommand{\ker}{\operatorname{ker}}
\renewcommand{\Re}{\operatorname{Re}}
\newcommand{\im}{\operatorname{im}}

\newcommand{\Mor}{\mathrm{Mor}}

\newcommand{\Ok}{{\calO_k}}
\newcommand{\OK}{{\calO_K}}
\newcommand{\OF}{{\calO_F}}
\newcommand{\OL}{{\calO_L}}
\newcommand{\Sch}{\mathrm{Sch}}
\newcommand{\OS}{\calO_S}

\newcommand{\dash}{\textendash}

\title{Serre's Tensor Construction and Moduli of Abelian Schemes}
\author{Zavosh Amir-Khosravi}
\address{Department of Mathematics\\California Institute of Technology\\
Pasadena, California}
\email{zavosh@caltech.edu}

\begin{document}
\begin{abstract}
Given a polarized abelian scheme with action by a ring, and a projective 
finitely presented module over that ring, Serre's tensor construction 
produces a new abelian scheme. We show that to equip these abelian schemes with 
polarizations it's enough to equip the projective modules with non-degenerate 
positive-definite hermitian forms. As an application, we relate certain moduli 
spaces of principally polarized abelian schemes with action by the ring of 
integers of a CM field. More specifically, we consider integral models of 
zero-dimensional Shimura varieties associated to compact unitary groups. We 
show that all abelian schemes in such moduli spaces are, \'etale locally over 
their base schemes, Serre constructions of CM abelian schemes with 
positive-definite 
hermitian modules. We also describe the morphisms between such objects in terms 
of morphisms between the constituent data, and formulate these relations as an 
isomorphism of algebraic stacks.
\end{abstract}
\maketitle
\section*{Introduction}

Let $R$ be a ring, possibly non-commutative, and free of finite rank over 
$\ZZ$. Let $(A,\iota)$ be an abelian scheme $A$ over a base $S$, with an 
injective ring homomorphism $\iota: R \hra \End_S(A)$ giving an $R$-action on 
$A$. Take $M$ to be a projective finitely presented right $R$-module. Serre's 
tensor construction associates to this data a new abelian scheme $M\otimes_R A$ 
over $S$, which is characterized by its functor of points $\Sch_{/S} \rar \Ab,\ 
T \mapsto M\otimes_R A(T)$ (Definition \ref{serredef}). The map $A \mapsto 
M\otimes_R A$ is functorial in $A$ and $M$, and preserves many desirable 
properties of $A$. This suggests the possibility of using it to relate families 
of abelian schemes. In order to do this, we first need to equip $M\otimes_R A$ 
with extra structures, in particular a polarization 
\cite{GIT,FaltingsChai} that is compatible with the $R$-action in the 
following sense.

Assume $R$ is equipped with a positive involution $r\mapsto r^*$ (Definition 
\ref{posinv}). Then the pair $(A,\iota)$ has a dual $(A^\vee,\iota^\vee)$, 
where $A^\vee$ is the dual abelian scheme of $A$, and $\iota^\vee(r) = 
\iota(r^*)^\vee$, for $r\in R$. A polarization $\lambda: A \rar A^\vee$ is said 
to be \textit{$\mathit{R}$-linear} if $\lambda \circ \iota(r) = \iota^\vee(r) 
\circ 
\lambda$. 

Assume $M_\QQ$ is free over $R_\QQ$. The dual module $M^\vee = \Hom_R(M,R)$ has 
a natural right $R$-module structure, with $r\in R$ acting on $f\in M^\vee$ by 
$(f\cdot r)(m) = r^* f(m)$. Then $R$-linear maps $h: M \rar M^\vee$ may be 
identified with sesquilinear forms $H: M\times M \rar R$ via 
$H(m,m')=h(m)(m')$. Such a map $h$ is called \textit{hermitian} if 
$H(m,m')=H(m',m)^*$, and \textit{non-degenerate} if it's an isomorphism. Since 
$M_\QQ\simeq R_\QQ^n$, we may identify $h$ with an element of $H_n(R_\QQ)$, the 
set of $n\times n$ hermitian matrices with entries in $R_\QQ$. Then 
$H_n(R_\QQ)\otimes \RR$ is a formally real Jordan algebra (Definition 
\ref{Jalg}) over $\RR$. We say $h$ is \textit{positive-definite} if its image 
under $H_n(R_\QQ) \subset H_n(R_\QQ)\otimes \RR$ is positive (Definition 
\ref{posdefn}). This notion does not depend on the choice of isomorphism $M_\QQ 
\simeq R_\QQ^n$ (Lemma \ref{Pprop2}). \\

\textbf{Theorem A.} \textit{Suppose $(A,\iota,\lambda)$ consists of an abelian 
scheme $A/S$, an $R$-action $\iota: R \hra \End(A)$, and $R$-linear 
polarization 
$\lambda: A \rar A^\vee$. Let $h: M \rar M^\vee$ be $R$-linear. The map 
$h\otimes \lambda: M\otimes_R A \rar M^\vee \otimes_R A^\vee$ is a polarization 
on $M\otimes_R A$ if and only if $h$ is a positive definite $R$-valued 
hermitian form.}\\

The above is Theorem \ref{mainthm} in the main text. That the abelian scheme 
dual to $M\otimes_R A$ is $M^\vee \otimes_R A^\vee$ is proved in Proposition 
\ref{duality}.  A special case of the theorem is due to Serre 
\cite[appx]{LauterSerre}, where $A$ is a specific elliptic curve in 
characteristic $p$. We also show that under some extra assumptions on $A$, if 
$\lambda$ is principal, then $h\otimes \lambda$ is principal if and only if $h$ 
is non-degenerate (Proposition \ref{mainprop}). For instance it's enough to 
assume $\End_S(A)$ is free over $R$.

We then apply the above result to the following moduli problem. Let $K$ be a CM-field of degree $2g$ over $\QQ$, $\Phi$ a CM-type for $K$, and $n$ a positive integer. Let $L\subset \CC$ be the reflex field of $(K,\Phi)$. To every locally noetherian scheme $S$ over $\Spec \OL$ we associate the category $\calM_\Phi^n(S)$ of triples $(A,\iota,\lambda)$ consisting of an abelian scheme $A$ of relative dimension $ng$ over $S$, an injective $\OK$-action $\iota$, and an $\OK$-linear principal polarization $\lambda$. We also require that $(A,\iota,\lambda)$ satisfy the \textit{ideal condition} $J_\Phi \Lie_S(A)=0$, where $J_\Phi$ is the kernel of
$$ \OK\otimes \OL \rar \prod_{\phi\in \Phi} \CC^{(\phi)},\ \ \ (\alpha \otimes 
\beta)\mapsto (\phi(\alpha)\cdot \beta)_\phi.$$
Morphisms of $\calM_\Phi^n(S)$ are $\OK$-linear isomorphisms of abelian schemes that preserve the polarizations (see Definition \ref{modspace}). 

The ideal condition is a refinement of the \textit{signature condition}. That 
says for $a\in \OK$, the characteristic polynomial of the induced action of 
$\iota(a)$ on $\Lie_S(A)$ should equal 
\begin{align} \label{detcon} \prod_{\phi \in \Phi}(T-\phi(a))^n \in \OL[T],\end{align}
viewed as a section of $\calO_S[T]$ using the structure morphism $\calO_L \rar 
\calO_S$. The ideal condition implies the signature condition (Corollary 
\ref{J->char}). If 
$n=1$ or the base $S$ has characteristic zero, the signature condition also 
implies the ideal condition. The ideal condition ensures that $\calM_\Phi^n$ is 
proper and smooth of relative dimension $0$ over $\calO_L$ 
(Theorem \ref{etprop}). When either $n=1$ or $K$ is quadratic imaginary, this 
is 
a theorem of B. Howard \cite{unitary1,unitary2}. The ideal condition $J_\Phi 
\Lie_S(A)=0$ allows us to generalize Howard's proof to all CM fields $K$.

We note that $\calM_\Phi^1$ is the moduli stack of abelian schemes with CM by $\OK$ of type $\Phi$. The stacks $\calM_\Phi^n$ are zero-dimensional versions of moduli spaces considered by S. Kudla and M. Rapoport (\cite{KR_11}, \cite{KR_09}), which are integral models of Shimura varieties of unitary type.

Using Theorem A, we can apply Serre's construction to the problem of 
constructing objects in $\calM_\Phi^n$. Let $\Herm_n(\OK)$ denote the category 
of pairs $(M,h)$ consisting of projective finitely presented $\OK$-modules $M$ 
of rank $n$, equipped with a positive-definite non-degenerate $\OK$-hermitian 
structure $h: M \rar M^\vee$. Then for $(M,h) \in \Herm_n(\OK)$ and 
$(A,\iota,\lambda) \in \calM_\Phi^1$, we construct the object
$$ (M,h) \otimes (A,\iota,\lambda) = (M\otimes_R A, \mathbb{1}_M \otimes \iota, h \otimes \lambda) \in \calM_\Phi^n.$$

To describe all such objects in $\calM_\Phi^n$ we define a tensor product 
of categories, by explicit generators and relations (Definition 
\ref{tensordef}). Then we construct the groupoid $ 
\Herm_n(\OK) \otimes_{\Herm_1(\OK)} \calM_\Phi^1(S)$ for each $S$. This is 
carried out in $\S 2$. We suggest the reader skip details of the otherwise 
intuitive abstract constructions in $\S2$, and consult the section as needed. 
Aside from the definitions the main result in $\S 2$ is Proposition 
\ref{morprop}, which by a combinatorial argument gives a concise presentation 
of a general morphism in the tensor product groupoid.

Serre's construction then induces a functor
\begin{align}\label{sigS} \Sigma_S: \Herm_n(\OK) \otimes_{\Herm_1(\OK)} 
\calM_\Phi^1(S) \rar \calM_\Phi^n(S).\end{align}

For $\Sigma_S$ to have any significance, $\calM^1(S)$ must be non-empty. We can 
assume this is the case after replacing $L$ by some finite extension, as long 
as $\calM_\Phi^1(\CC)\neq \emptyset$. We show $\calM_\Phi^1(\CC)\neq \emptyset$ 
if $K/F$ is ramified at any finite prime 
(Theorem \ref{existence}). Also if $\calM_\Phi^n(\CC)\neq \emptyset$ and $n$ is 
odd, we have $\calM_\Phi^1(\CC)\neq \emptyset$ (Proposition 
\ref{nodd}). In general, we assume $\calM_\Phi^1(\CC)\neq \emptyset$.

The functor $S \mapsto \Herm_n(\OK) \otimes_{\Herm_1(\OK)} \calM_\Phi^1(S)$ 
defines a separated presheaf on the big \'etale site over $\Spec \OL$. Letting 
$\Herm_n(\OK) \otimes \calM_\Phi^1$ denote the associated stack, Serre's 
construction induces a morphism
$$ \Sigma: \Herm_n(\OK) \otimes \calM_\Phi^1 \rar \calM_\Phi^n.$$

\textbf{Theorem B.} Assume that $\calM_\Phi^1(\CC)\neq \emptyset$. Then:
\begin{enumerate}
	\item $\Sigma$ is an isomorphism of $\OL$-stacks.
	\item If $S=\Spec(k)$, for $k$ an algebraically closed field, then 
	$\Sigma_S$ is an equivalence of categories.
	\item If $S$ is locally noetherian over $\Spec \OL$, each object of 
	$\calM_\Phi^n(S)$ is \'etale locally on $S$ in the image of $\Sigma_S$ 
	from (\ref{sigS}).
\end{enumerate}
The above is Theorem \ref{mainiso} in the paper. The proof begins by showing 
the functor $\Sigma_S$ is fully faithful (Proposition \ref{fullfaith}). 
This is done by comparing general forms of morphisms in $\calM_\Phi^n$ 
and the tensor product. 
Essential surjectivity of $\Sigma$ is proved on the stalks 
of geometric points, first in characteristic zero by explicit construction, 
then extended to characteristic $p$ by smoothness of $\calM_\Phi^n$ over $\Spec 
\OL$.

For a restatement of the content of Theorem $B$ without the language of higher 
categories see Theorem \ref{vanilla} in the text.

\subsection*{Acknowledgements}

Much of the content in this article is part of the author's thesis, written at 
the University of Toronto under the supervision of Stephen S. Kudla. I thank 
Professor Kudla for suggesting the problem and for subsequent encouragements. I 
thank Brian Conrad and Florian Herzig for looking over an early draft, and 
offering helpful comments and suggestions.

I also thank Ben Howard, who read the first version carefully, spotting mistakes and suggesting improvements. In particular the ideal condition $J_\Phi \Lie_S(A)=0$ in the definition of $\calM_\Phi^n$ was suggested by him. I express gratitude to the anonymous referee for patiently reading the article, uncovering many mistakes, and suggesting significant improvements and simplifications, especially in the proofs of the final sections.
\tableofcontents

\section{Serre's Construction}

In this section we recall the Serre tensor construction, then establish some basic properties of the abelian schemes arising from it, such as the possible homomorphisms between them, their Tate modules, Lie algebras, and their dual abelian schemes. We then study the polarizations on such abelian schemes.

Let $R$ be a ring, possibly non-commutative, and free of finite rank over 
$\ZZ$. An \textit{abelian scheme with an $R$-action} is a pair $(A,\iota)$, 
where $A$ is an abelian scheme over some base $S$, and $\iota: R \rar 
\End_S(A)$ is a ring homomorphism. In subsequent sections we will further 
assume that $\iota$ is \textit{injective}.

\begin{defn}\label{serredef}Let $(A,\iota)$ be an abelian scheme with an 
$R$-action, and $M$ a projective finitely presented right $R$-module. The 
\textbf{Serre tensor construction} is the 
abelian scheme $M\otimes_R A$ that represents (\cite[\S7]{Conrad_revisited}) 
the 
group-functor 
\[ M\otimes_R A: \mathrm{Sch}_{/S} \lra \mathrm{Ab},\ \ \ T \mapsto M\otimes_R 
A(T).\]
\end{defn}

Suppose $P$ is another ring, free of finite rank over $\ZZ$, and the right 
$R$-module $M$ is also a left $P$-module, such that the actions of $P$ and $R$ 
commute. Then $P$ acts on $M\otimes_R A$ via the $M$ factor. In particular 
if $R$ is commutative then $M\otimes_R A$ has an $R$-action.

Let $(A,\iota)$, $(B,\jmath)$ be abelian schemes with $R$-actions. An 
\textit{$R$-linear} homomorphism $\phi: A \rar B$ is a homomorphism of abelian 
schemes satisfying $\phi\circ \iota(r) = \jmath(r) \circ \phi$ for all $r\in 
R$.  If $f: M \rar N$ is a homomorphism of projective finitely presented right 
$R$-modules, and $\phi: A \rar B$ an $R$-linear homomorphism of abelian 
schemes, by $f\otimes \phi: M\otimes_R A \rar N\otimes_R B$ we denote the map 
given on $T$-valued points by
$$ (f\otimes \phi)_T: M\otimes_R A(T) \rar N\otimes_R B(T),\ \ \ m\otimes a \mapsto f(m)\otimes \phi(a),\ \ \ T\in \Sch_{/S}.$$
We often hide canonical isomorphisms, e.g. we write $A^n=R^n \otimes_R A$. We 
also write $f_A$ for $f\otimes \mathbbm{1}_A$.

\subsection{Homomorphisms} The key proposition is the following.
\begin{prop}\label{homprop}
Let $A$ be an abelian scheme over $S$, with action by a ring $R$, and suppose $M$ is a projective finitely presented right $R$-module. Let $B$ be another abelian scheme over $S$, with action by a ring $P$, and $N$ a projective and finitely presented right $P$-module. 
\begin{enumerate}
\item[(a)] There is a canonical isomorphism of abelian groups
	\[ \Psi: N\otimes_P \Hom_S(A,B) \otimes_R M^\vee \cong \Hom_S(M\otimes_R A,N\otimes_P B),\]
	mapping a pure tensor $n\otimes \phi \otimes f$ to the morphism given on 
	$T$-valued points by $$\Psi(n\otimes \phi \otimes f)_T: M\otimes_R A(T) 
	\rar N\otimes_P B(T),\ \ \ m\otimes a \mapsto n\otimes \phi(f(m)a),\ \ \ 
	T\in \Sch_{/S}.$$
	\item[(b)] Suppose $P=R$ and $M$, $N$ are $R$-bimodules, so that $M\otimes_R A$, $N\otimes_R B$ acquire $R$-actions. The above isomorphism, restricted to $R$-linear homomorphisms, gives a canonical isomorphism
	\[ \Psi: N\otimes_R \Hom_R(A,B) \otimes_R M^\vee \cong \Hom_R(M\otimes_R A,N\otimes_R B).\]
	\item[(c)] With $M$, $N$, $R$ as in (b), suppose $R$ is moreover commutative. Then there's a canonical isomorphism of $R$-modules
	\[ \Hom_R(M,N) \otimes_R \Hom_R(A,B) \cong \Hom_R(M\otimes_R A,N\otimes_R B),\]
	mapping $h\otimes \phi$ to the morphism given on $T$-valued points by $h\otimes \phi_T$, for $T\in \Sch_{/S}$.
\end{enumerate}
\end{prop}
\begin{proof} For part (a), the statement is obvious if $M$ and $N$ are free. 
For the general case, pick a finitely presented projective right $R$-module 
$M'$, resp. $P$-module $N'$, so that $M_0=M\oplus M'$ and $N_0 = 
N\oplus N'$ are free of finite rank. Then the isomorphism 
$$N_0 \otimes_P \Hom_S(A,B) \otimes_R M_0^\vee \cong \Hom_S(M_0\otimes_R A, N_0\otimes_P A)$$ 
decomposes into a direct sum of four morphisms of abelian groups, which all 
must be isomorphisms, one of which coincides with the morphism in the 
statement. The explicit form of the map may be checked by following through the 
canonical isomorphisms involved.

The other parts are similar. For $(b)$, first assume $M\simeq R^m$ and $N\simeq R^n$, so that $\Hom(M\otimes_R A, N\otimes_R B)$ may be identified with $M_n(\Hom(A,B))$, the additive group of $m\times n$ matrices with entries in $\Hom(A,B)$. The claim becomes equivalent to the fact that $R$-linear elements of $M_n(\Hom(A,B))$ correspond to matrices with $R$-linear entries. The general case may be deduced by picking complementary projective modules as in part (a).

For $(c)$ note that the left and right $R$-actions on $\Hom_R(A,B)$ agree by definition of $R$-linearity. The claim then follows from $(b)$ and the associativity of tensor products of $R$-bimodules, plus the fact that the canonical morphism $M^\vee \otimes_R N \rar \Hom_R(M,N)$ is an isomorphism since $M$ and $N$ are finitely presented and projective.
\end{proof}

\subsection{Lie algebra and Tate module.}
	
The following lemma says taking the Tate module or Lie algebra of an abelian scheme commutes with applying Serre's construction.  

\begin{lemma}\label{lielem}
Let $A$ be an abelian scheme over a base $S$, equipped with an action $\iota: R 
\rar \End_S(A)$ by a ring $R$. Suppose $M$ is a projective finitely presented 
right $R$-module. There is then a canonical isomorphism of group schemes
\[ T_l(M\otimes_R A) \cong M\otimes_R T_l(A),\]
as well as a canonical isomorphism of $\calO_S$-modules
\[ \Lie_S(M\otimes_R A) \cong M\otimes_R \Lie_S(A).\]
\end{lemma}

\begin{proof}
For any positive integer $N$, the sequence
$$ 0 \rar M\otimes_R A[N] \rar M\otimes_R A \overset{N}{\rar} M\otimes_R A \rar 0$$
is exact since $M$ is a flat $R$-module. The first claim follows by passing to the limit. 

For the second assertion we use the functorial description of $\Lie_S(A)$ given by
\[ \Lie_S(A) (U) = \ker(A(U[\epsilon]) \rar A(U)),\]
for $U\subset S$ \cite[Exp II, 3.9]{SGA3TomeI}. Here $A(U[\epsilon]) \rar A(U)$ is induced by $U \rar U[\epsilon]$, which is constructed as follows. $U[\epsilon] = U\times_{\Spec(\ZZ)} \ZZ[\epsilon]$, where $\ZZ[\epsilon]$ is the ring of dual numbers. The map $U \rar U[\epsilon]$ comes from applying the fibre product functor $U\times_{\Spec \ZZ} \dash $ to the morphism $\Spec \ZZ \rar \Spec \ZZ[\epsilon]$, and the latter corresponds to the ring homomorphism $\ZZ[\epsilon] \rar \ZZ$ that sends $\epsilon$ to $0$. The claim follows again from the fact that $M\otimes_R -$ preserves kernels by flatness.
\end{proof}

\subsection{The dual abelian scheme} 

Recall that the dual $M^\vee=\Hom_R(M,R)$ of a right $R$-module $M$ is 
naturally a left $R$-module, with $(r \cdot f)(m) = r\cdot f(m)$, $r\in R$, 
$f\in M^\vee$, $m\in M.$

\begin{defn}\label{posinv}
A \textit{positive involution ring} $(R,*)$ is a ring $R$, free of finite rank over $\ZZ$, equipped with an involution $r\mapsto r^*$ such that $(a,b) \mapsto \Tr_{R_\QQ/\QQ}(ab^*)$ is positive-definite on $R_\QQ$. 
\end{defn}

We assume $(R,*)$ is a positive involution ring. Then $M^\vee$ is a right 
$R$-module via
\begin{align}\label{dualmodR} (f\cdot r)(m) = r^* f(m).\end{align}
By $f^\vee$ we denote the dual of $f: M \rar N$. The map $f \mapsto f(1)^*$ 
defines a canonical isomorphism of left $R$-modules 
$R^\vee \rar R$. 

Let $\scrP=\scrP_R$ denote the category of projective finitely presented right 
$R$-modules. The map $M \mapsto M^\vee$ defines a contravariant functor from 
$\scrP$ to itself. We write $(M^\vee)^\vee=M$ by abuse of notation.

Fixing a base scheme $S$, by $\scrA = \scrA(S)$ we denote the category of 
abelian schemes and group scheme homomorphisms over $S$. The map $A \mapsto 
A^\vee$ is a contravariant functor from $\scrA$ to itself. We write 
$(A^\vee)^\vee = A$ by abuse.

Let $\scrA_R=\scrA_R(S)$ denote the category of pairs $(A,\iota)$, where $A\in 
\scrA$ is equipped with an $R$-action. The morphisms in $\scrA_R$ are required 
to be $R$-linear. Each $(A,\iota)$ has a dual 
$(A^\vee,\iota^\vee)$ 
with $\iota^\vee(r) = \iota(r^*)^\vee$, and $(A,\iota) \mapsto 
(A^\vee,\iota^\vee)$ is a contravariant functor from $\scrA_R$ to itself.

Let $\calS : \scrP \times \scrA_R \rar \scrA$ be the functor induced by Serre's 
construction sending $(M,A)$ to $M\otimes_R A$.
 
\begin{prop} \label{duality}The following diagram commutes up to canonical isomorphism:
\begin{align*}
\xymatrix{ \scrP \times \scrA_R \ar[r]^-{\calS} \ar[d]_{\vee \times \vee} & \scrA \ar[d]^\vee \\
\scrP \times \scrA_R \ar[r]_-{\calS} & \scrA. }
\end{align*}
In other words, for $M\in \scrP$, $(A,\iota)\in \scrA_R$, and $f\in \Mor(\scrP)$, $\phi\in \Mor(\scrA_R)$, we have
$$ (M\otimes_R A)^\vee \cong M^\vee \otimes_R A^\vee, \ \ \ (f\otimes \phi)^\vee \cong f^\vee \otimes \phi^\vee.$$
The isomorphism $\Phi=\Phi_{M,A}: M^\vee \otimes_R A^\vee \rar (M\otimes_R A)^\vee$ is characterized as follows. For $T\in \scrA$, and
$g\otimes t \in M^\vee \otimes_R \Hom_S(T,A^\vee)$, the map $\Phi_T(g \otimes t)\in \Hom_{S}(T, (M \otimes_R A)^\vee)$ is the dual of the homomorphism $M\otimes_R A \rar T^\vee$ given on $U$-valued points by
$$ M\otimes_R A(U) \rar T^\vee(U),\ \ \  m\otimes u \mapsto t^\vee \circ \iota(g(m))\circ u, \ \ \ U\in \Sch_{/S}.$$

\end{prop}
\begin{proof} Let $T$ be an abelian scheme over $S$. We have canonical 
isomorphisms:
\begin{align*} \Hom_S(T, M^\vee \otimes_R A^\vee) & \cong M^\vee \otimes_R 
\Hom_S(T,A^\vee) && (\text{by Proposition }\ref{homprop}(a) )\\ 
& \cong 
\Hom_S(A,T^\vee)\otimes_R 
M^\vee & &(f\otimes 
\phi\mapsto \phi^\vee \otimes f)  \\ & \cong \Hom_S(M\otimes_R A,T^\vee) & & 
(\text{by Proposition }\ref{homprop}(a) )\\ & 
\cong \Hom_S(T,(M\otimes_R A)^\vee). & & \text{(duality)}
\end{align*}
Letting $T=M^\vee\otimes_R A^\vee$, the canonical 
morphism $\Phi: M^\vee\otimes_R 
A^\vee \rar (M\otimes_R A)^\vee $ corresponds to the identity element in
$$ \Hom_S(M^\vee \otimes_R A^\vee, M^\vee \otimes_R A^\vee) \cong \Hom_S(M^\vee \otimes_R A^\vee, (M\otimes_R A)^\vee),$$
and its inverse $ (M\otimes_R A)^\vee \rar M^\vee \otimes_R A^\vee$ corresponds 
to the identity in
$$ \Hom_S((M\otimes_R A)^\vee,M^\vee \otimes_R A^\vee) \cong \Hom_S( (M\otimes_R A)^\vee, (M\otimes_R A)^\vee),$$
which comes from setting $T=(M\otimes_ R A)^\vee$. 

The explicit form of $\Phi$, as well as the relation 
$(f\otimes \phi)^\vee=f^\vee \otimes \phi^\vee$, may be checked by following 
through these 
isomorphisms carefully.
\end{proof}

\subsection{Polarizations} We recall some basic definitions and facts about 
polarizations of abelian schemes. Let $A$ be an abelian scheme over a base $S$. 
The Poincar\'e correspondence $\scrP_A$, is a universal line bundle on 
$A\times_S A^\vee$ that induces, for any abelian scheme $B/S$, a canonical 
isomorphism of groups
\[ \Hom_S(B,A^\vee) \cong \Corr_S(A,B),\ \ \ (\phi: B \rar A^\vee) \mapsto (\mathbbm{1}_A \times \phi)^* (\scrP_A).\]
Here $\Corr_S(A,B)$ denotes the group of correspondences on $A\times_S B$ \cite[I.1.7]{FaltingsChai}. 

Let $\Delta: A \rar A \times_S A$ be the diagonal. For a morphism $f: A \rar A^\vee$ of abelian schemes, let $\calL_f$ denote the correspondence $(1\times f)^* \scrP_A$ on $A\times_S A$, and $\scrL_f = \Delta^*(\calL_f)$ the associated line bundle on $A$. If $g: B \rar A$ is a homomorphism of abelian schemes, then the pullbacks of $\scrL_f$ and $\calL_f$ under $g$ and $g\times g$ correspond to the map $g^\vee \circ f \circ g: B \rar B^\vee$. In other words:

\begin{equation}\label{pullbackrels} g^*\scrL_f = \scrL_{g^\vee \circ f \circ g}\ \ ,\ \ \ (g\times g)^* \calL_f = \calL_{g^\vee \circ f \circ g}.
\end{equation}

If $\calL$ is a correspondence on $A\times_S A$, its associated map $\lambda: A \rar A^\vee$ is \textit{symmetric} if and only if $\calL$ is a symmetric correspondence, meaning $s^*(\calL) \simeq \calL$ with $s: A \times A \rar A \times A$ the coordinate flip map. 

A \textit{polarization} on an abelian variety $A_0$ is a symmetric homomorphism 
$\lambda: A_0 \rar A_0^\vee$ associated to a correspondence $\calL_\lambda$ as 
above, such that the line bundle $\scrL_\lambda=\Delta^*(\calL_\lambda)$ on 
$A_0$ is ample. A polarization on the abelian scheme $A$ is a symmetric 
homomorphism $\lambda: A \rar A^\vee$ such that for every geometric point 
$\ol{s} \rar S$, $\lambda_{\ol{s}}: A_{\ol{s}} \rar {A_{\ol{s}}}^\vee$ is a 
polarization of abelian varieties. A \textit{principal} polarization is one 
that is an isomorphism.

The choice of a polarization on an abelian scheme $A$ induces a \textit{Rosati involution} $\phi\mapsto \rho(\phi)$ on $\End^0(A) = \End(A)\otimes \QQ$, determined by the commutativity of the following diagram in the isogeny category:
 \[ \begin{xymatrix}{ A \ar[r]^\lambda \ar[d]_{\rho(\phi)} & A^\vee \ar[d]^{\phi^\vee} \\
 	A \ar[r]^\lambda & A^\vee. }
 \end{xymatrix}
 \]

Let $(R,*)$ be a ring with involution, and $(A,\iota)$ an abelian scheme with 
$R$-action. Then the dual $A^\vee$ has an $R$-action $\iota^\vee$ given by 
$\iota^\vee(r) = \iota(r^*)^\vee$. A 
polarization $\lambda: A \rar A^\vee$ is $R$-linear if and only if $\lambda 
\circ \iota(r) = \iota(r^*)^\vee \circ \lambda$ for all $r\in R$. By the above 
diagram, this is equivalent to $\iota(r^*) = \rho(\iota(r))$, so that $\lambda$ 
is $R$-linear if and only if $\iota$ maps $*$ to the Rosati $\rho$.

Let $(A,\iota,\lambda)$ be as above, with $\lambda$ an $R$-linear polarization, 
and denote by $\calL = \calL_\lambda$ the correspondence associated to it. The 
behaviour of the pullback of $\calL$ under the product $f\times g$ of various 
maps $f,g\in \End(A)$ is described as follows.

\begin{prop}
	The map $l: \End(A) \times \End(A) \rar \Corr(A,A)$ given by $l(x,y) = (x\times y)^* \calL$ satisfies the linearity relations
	\begin{align*}
	& l(x+y,z) = l(x,z) \otimes l(y,z) \\
	& l(x,y+z) = l(x,y) \otimes l(x,z)\\
	& l(x,y\circ \iota(r)) = l(x\circ \iota(r^*), y),
	\end{align*}
	for all $x,y,z \in \End(A)$, and $r\in R$.
\end{prop}

\begin{proof}
A corollary of the theorem of the cube \cite[p.59]{Mum} states that if $f,g,h: X \rar Y$ are maps of abelian varieties, and $\scrL$ is a line bundle on $B$, then:
\[ (f+g+h)^*\scrL \cong (f+g)^* \scrL \otimes (g+h)^* \scrL \otimes (f+h)^* \scrL \otimes f^* \scrL^{-1} \otimes g^* \scrL^{-1} \otimes h^* \scrL^{-1}.
\]
The first property follows from the above applied to $f=(x\times 0)$, $g=(y\times 0)$, $h=(0\times z)$, and $\scrL = \calL_\lambda$, with the second being similar. The third property follows from the $R$-linearity of $\lambda$.
\end{proof}
	
Let $(R,*)$ be a ring with an involution, and $M$ a projective finitely 
presented right $R$-module. A $\ZZ$-bilinear form $F: M\times M \rar R$ is 
called \textit{sesquilinear} if $F(mr,nr')=r^*F(m,n)r'$ for all $m,n\in M$, and 
$r,r'\in R$. By the tensor-hom adjunction formula, sesquilinear forms $F$ are 
in bijection with linear maps $f \in \Hom_R(M,M^\vee)$ via $f(m)(n) = F(m,n)$. 
A sesquilinear form $F$ is called \textit{hermitian} if $F(m,n)^*=F(n,m)$ for 
all $m,n\in M$. 

Let $p: R^n \thrar M$ be a fixed presentation of $M$, and $e_1,\cdots, e_n$ the 
image of the standard basis elements. A sesquilinear form $F$ on $M$ is 
hermitian if and only if the $n\times n$ matrix with entries $F(e_i,e_j)$ is 
hermitian. That is, if and only if $F(e_i,e_j)=F(e_j,e_i)^*$ 
for $1 \leq i,j\leq n$.

If $p_A = p\otimes \mathbb{1}_A: A^n \rar M\otimes_R A^n$, then $\psi = 
p_A^\vee \circ 
(f\otimes \lambda) \circ p_A \in  \End_S(A^n, (A^\vee)^n)$. The map $\psi$ is 
determined by an $n\times n$ matrix of morphisms $\psi_{ij}: A \rar A^\vee$. 

\begin{prop}\label{symprop}
	With $f\otimes \lambda$ and $\psi$ as above, the following are equivalent:
	\begin{enumerate}
		\item $f \otimes \lambda$ is symmetric.
		\item $\psi$ is symmetric.
		\item $\psi_{ij} = (\psi_{ji})^\vee$ for $1 \leq i,j \leq n$.
		\item $f$ is hermitian.
	\end{enumerate}
\end{prop}

\begin{proof} 
This is left to the reader as an exercise in linear algebra. \end{proof}

Now we recall Jordan algebras and the notion of positivity in a formally real Jordan algebra over $\RR$. We relate positivity in certain matrix Jordan algebras with the usual notion of a positive-definite matrix. Then we define positive-definite $R$-hermitian structures $h: M \rar M^\vee$, which correspond to polarizations on $M\otimes_R A$. The reference for this material is \cite{Braun1966} and \cite{McCrimmon2004}.

\begin{defn} \label{Jalg} Let $k$ be a field, with $\text{char}(k)\neq 2$, and let $R$ be an algebra over $k$, not necessarily associative, with multiplication denoted by $\circ$. Then $(R,\circ)$ is called a \textbf{Jordan algebra} if it is commutative, and  
\[ (u\circ u)\circ (u\circ v) = u \circ ((u\circ u)\circ v),\ \ \ \forall u,v\in R.\]
A Jordan algebra $R$ is called \textbf{formally real} if for all $u,v \in R$,  $$u\circ u + v \circ v = 0  \iff u=v=0.$$
\end{defn}

Any associative $k$-algebra $R$ can be turned into a Jordan algebra $(R,\circ)$ by setting
\[ x \circ y = \frac{1}{2} (xy + yx).\]

We restrict to the case where $k$ is subfield of $\RR$, so that it makes sense to speak of positive elements of $k$. Let $R$ be a finite associative $k$-algebra equipped with a positive involution $r \mapsto r^*$, so that $(x,y) \mapsto \Tr_{R/k}(y^*x)$ is positive definite. The elements of $R$ fixed by the involution are called \textit{symmetric}. The positivity of the involution implies that the subalgebra of symmetric elements $S\subset R$ is formally real. For any $n>0$, the matrix algebra $M_n(R)$ inherits a positive involution $X\mapsto X^*$ from $R$, given by $(X_{ij})\mapsto (X_{ji}^*)$. Its symmetric elements are the formally real Jordan algebra $H_n(R)$ of $n\times n$ hermitian matrices.

\begin{defn} \label {Jpos} An element $u$ of a formally real Jordan algebra 
over the real numbers $\RR$ is called \textbf{positive}, and denoted $u>0$, if 
all the eigenvalues of the $\RR$-linear map $L_u(v)= u\circ v$ are positive.
\end{defn}

If $R$ is any one of: the real numbers $\RR$, complex numbers $\CC$, or the 
standard quaternions $\HH$, it can be considered as a Jordan algebra over 
$\RR$, with a positive involution given by the identity map on $\RR$, complex 
conjugation on $\CC$, and the standard involution on $\HH$. In all three cases, 
the matrices in $H_n(R)$ are unitarily diagonalizable. For a matrix $X\in 
H_n(R)$, the eigenvalues of the operator $L_X: H_n(R) \rar H_n(R),\ L_X(Y) = 
X\circ Y = 
\frac{1}{2} (XY+YX)$ are $\frac{1}{2}(d_i + d_j)$ where $d_i$ are the ordinary 
eigenvalues of $X$ as a matrix. It follows that $X > 0$ in $H_n(R)$ if and 
only if $X$ is a positive definite matrix in the usual sense.

Now suppose $R=K$ is a CM field, with maximal totally real subfield $F$. Then 
$K$ is an algebra over $F$, with complex conjugation defining a positive 
involution. The formally real Jordan algebra $H_n(K)\otimes \RR$ over $\RR$ is 
isomorphic to a product of algebras $H_n(\RR)$ and $H_n(\CC)$, one for each 
embedding of $F$ in $\RR	$. A matrix $X\in H_n(K)$ is positive in 
$H_n(K)\otimes 
\RR$ if and only if it is positive in each factor of the product. It follows 
that $X >0$ if and only if the eigenvalues of $X$ are totally positive 
algebraic numbers.

If $A$ is an abelian variety with a polarization, $\End(A)_\QQ$ is equipped 
with a Rosati involution. By a \textit{symmetric} element of $\End(A)$, resp. 
$\End(A)_\QQ$, we mean one that is fixed by the Rosati involution. We denote 
such elements 
by $\End(A)^{\mathrm{sym}}$, resp. $\End(A)_\QQ^{\mathrm{sym}}$.

We recall a result characterizing ample line bundles on abelian varieties. The 
following theorem is quoted from \cite[p.208]{Mum}, where the term 
\textit{totally positive} is used for what we call positive.

\begin{thm}[\textbf{Ampleness Criterion}] 
For $r\in \End(A)^{\mathrm{sym}}$, the line bundle $\scrL_{\lambda \circ r} \in \Pic(A)$ is ample if and only if $r$ is \textit{positive} in the formally real Jordan algebra $\End(A) \otimes \RR$. 
\label{amplcrit}
\end{thm}

\begin{cor}\label{ampcor}
Let $S$ be a connected scheme, $A$ an abelian scheme over $S$, $\lambda: A \rar 
A^\vee$ a polarization, and $r\in \End_S(A)$. Then $\lambda \circ r$ is a 
polarization if and only if $r$ is symmetric and positive in the formally real 
Jordan algebra $\End_S(A)\otimes \RR$.
\end{cor}
\begin{proof}
Using the rigidity lemma of \cite[Ch. 6]{GIT}, one easily reduces to the case 
where $S$ is the spectrum of an algebraically closed field. Let $r \mapsto r^*$ 
denote the Rosati involution of $\lambda$. For $r\in \End_S(A)$, $\lambda \circ 
r$ is symmetric if and only if $r$ is symmetric: this follows from $(\lambda 
\circ r)^\vee = r^\vee \circ \lambda = \lambda \circ r^*$, and the fact that 
$\lambda$ is an isogeny. Now suppose $r$ is symmetric. 
By Mumford's 
ampleness criterion, the line bundle $\scrL_{\lambda \circ r}$ corresponding to 
$\lambda 
\circ r$ is ample if and only if $r$ is positive in $\End(A) \otimes \RR$. 
Hence $r$ is symmetric and positive in $\End_S(A)\otimes \RR$, if and only if 
$\lambda \circ r$ is symmetric and $\scrL_{\lambda \circ r}$ is ample.
\end{proof}

\begin{defn}\label{Pdef}
A positive involution ring $(R,*)$ is said to satisfy \textbf{property (P)}, if for every matrix $Q$ in $GL_n(R_\QQ)$ the hermitian matrix $Q^*Q \in H_n(R_\QQ)$ is positive in the formally real Jordan algebra $H_n(R)\otimes \RR$.
\end{defn}

The following lemma shows that property (P) holds in all cases of interest.

\begin{lemma}\label{Pprop} Let $(R,*)$ be a positive involution ring. Suppose an abelian scheme $A/S$ exists that admits an $R$-action, with an $R$-linear 
polarization $\lambda$. Then $R$ necessarily satisfies property (P).
\end{lemma}
\begin{proof}
Since for any point $s$ of the base scheme $S$ the map $\End_S(A) \rar 
\End_{k(s)}(A_s)$ is injective, we can assume $A$ is an abelian variety. Let 
$Q\in GL_n(R_\QQ)$. After multiplying by a positive integer we can assume $Q$ 
has entries in $R$, and so defines an $R$-linear isogeny $\phi: A^n \rar A^n$. 
Since the $R$-linear map $\lambda^n: A^n \rar (A^\vee)^n$ is a polarization on 
$A^n$, the line bundle $\scrL_{\lambda^n}$ is ample, therefore the line bundle 
$\scrL_{\phi^\vee \circ \lambda^n \circ \phi}=\phi^* \scrL_{\lambda^n}$ is also 
ample, as $\phi$ is an isogeny. Since $\lambda$ is $R$-linear, $\phi^\vee \circ 
\lambda^n \circ \phi = \lambda^n \circ \psi$, where $\psi: A^n \rar A^n$ is 
given by the hermitian matrix $Q^*Q$ with coefficients in $R$. Then by the 
ampleness 
criterion, the symmetric element $\psi$ is positive in the formally real Jordan 
algebra $H_n(R_\QQ)\otimes \RR \subset \End(A^n)\otimes \RR$, where it is 
identified with $Q^*Q$.
\end{proof}

A projective finitely presented module $M$ over a Dedekind domain is automatically a lattice in $M_\QQ$. For general rings $R$, we make the following definition.

\begin{defn}\label{lattice}A \textbf{lattice} $M$ over $R$ is a projective finitely presented right $R$-module $M$ such that $M_\QQ$ is free over $R_\QQ$. 
\end{defn}

The key positivity notion is as follows.

\begin{defn}\label{posdefn} Suppose $(R,*)$ satisfies property (P). A hermitian lattice $(M,h)$ is \textbf{positive definite} if for some isomorphism $\eta: R^n_\QQ \rar M_\QQ$, the standard matrix $T\in H_n(R_\QQ)$ of the map $\eta^\vee \circ h_\QQ \circ \eta$ is positive in the formally real Jordan algebra $H_n(R)\otimes \RR$.
\end{defn}

The following lemma shows that the above definition does not depend on the choice of $\eta$.

\begin{lemma}\label{Pprop2}
	Let $(R,*)$ be a positive involution ring with property (P). Let $Q\in GL_n(R_\QQ)$ and $T\in H_n(R_\QQ)$. Then $T>0$ in $H_n(R_\QQ)\otimes \RR$ if and only if $QTQ^*>0$.
\end{lemma}

\begin{proof}
	Let $J_1$ denote the Jordan algebra $H_n(R_\QQ)$, and $J_{QQ^*}$ the algebra with the same underlying abelian group as $J_1$ and product defined by 
		\[ X\circ_{QQ^*} Y = \frac{1}{2} \left(X(QQ^*)^{-1}Y + Y(QQ^*)^{-1}X\right).\]
    Then $J_{QQ^*}$ is a formally real Jordan algebra with unit element $QQ^*$, and the map $T\mapsto QTQ^*$ is an isotopy of Jordan algebras $J_1\rar J_{QQ^*}$ \cite[p.14]{McCrimmon2004}. It can be extended $\RR$-linearly to an isotopy $J_1\otimes \RR \rar J_{QQ^*}\otimes \RR$ of formally real Jordan algebras over $\RR$. 
	
	The set of positive elements of a formally real Jordan algebra $J$ over $\RR$ is an open convex cone, identified with the connected component of the identity in the units $J^\times$ of $J$ \cite[p.18]{McCrimmon2004}. An isotopy of Jordan algebras preserves the positive cone.
	
	The positive cone of $J_{QQ^*}\otimes \RR$, which is the same topological space as $J_1 \otimes \RR$, is the connected component of its identity element $QQ^*$. Since $(R,*)$ satisfies property (P), $QQ^*$ is positive in $J_1\otimes \RR$, therefore it lies in the same connected component as the identity element of $J_1\otimes \RR$. Thus the isotopy $T \mapsto QTQ^*$ maps the positive cone of $J_1\otimes \RR$ to itself, and so $T>0$ if and only if $QTQ^*>0$.
\end{proof}

The main result of this section (Theorem A from the introduction), is as follows.

\begin{thm} \label{mainthm}
	Let $(R,*)$ be a positive involution ring, $(A,\iota)$ an abelian scheme with an $R$-action, and $M$ a lattice over $R$. Suppose $\lambda: A \rar A^\vee$ is an $R$-linear polarization and $h: M \rar M^\vee$ is an $R$-linear map. Then $h\otimes \lambda: M \otimes_R A \rar M^\vee \otimes_R A^\vee$ is a polarization if and only if $(M,h)$ is hermitian and positive definite.
\end{thm}

\begin{proof}
By Proposition \ref{symprop}, $h \otimes \lambda$ is symmetric if and only if $h$ is hermitian, so we can assume this is the case. Such a symmetric map is by definition a polarization if and only if it is a polarization of abelian varieties on geometric fibres. Since Serre's construction commutes with fibre products, $h\otimes \lambda$ is a polarization if and only if $(h\otimes \lambda)_s \cong  h\otimes \lambda_s$ is a polarization for all geometric points $s \rar S$. Thus we can assume $A$ is an abelian variety over an algebraically closed field.
	
We fix an isomorphism $\eta: R^n_\QQ \rar M_\QQ$, through which we identify $M\subset M_\QQ$ with its pre-image in $R^n_\QQ$. Let $T\in H_n(R_\QQ)$ be the matrix of $\eta^\vee \circ h_\QQ\circ \eta$ as in Definition \ref{posdefn}. Since the inclusion $M\subset R^n_\QQ$ becomes an isomorphism after tensoring with $\QQ$, there is a positive integer $k$ such that $kR^n \subset M$. Let $\kappa: R^n \rar M$ denote multiplication by $k$. Using the identification $A^n=R^n \otimes_R A$, we obtain an isogeny $\kappa_A = \kappa \otimes \mathbbm{1}_A: A^n \rar M\otimes_R A$. Now let $f = \kappa^\vee \circ h \circ \kappa$ and consider the map $f\otimes \lambda: A^n \rar (A^\vee)^n$. Since $h\otimes \lambda$ is symmetric, so is $f\otimes \lambda = \kappa_A^\vee \circ (h\otimes \lambda) \circ \kappa_A$. The matrix of $f$ is a positive integer multiple of $T$, so $(R^n,f)$ is positive-definite if and only if $(M,h)$ is. Since $\kappa_A$ is an isogeny, the line bundle $\kappa_A^*(\scrL_{h \otimes \lambda}) = \scrL_{f\otimes \lambda}$ is ample if and only if $\scrL_{h\otimes \lambda}$ is ample. It follows that it's enough to prove the theorem for $(R^n,f)$ in place of $(M,h)$.
	
Now $f\otimes \lambda: A^n \rar (A^\vee)^n$ factors as $\lambda^n \circ f_A$, where $\lambda^n: A^n \rar (A^\vee)^n$ is the product polarization of $A^n$ obtained from $\lambda$. By Corollary \ref{ampcor}, $\lambda^n \circ f_A$ is a polarization if and only if $f_A \in \End(A^n)$ is positive in the formally real Jordan algebra $\End(A^n)\otimes \RR$. It remains to show this is the case if and only if $f$ is positive definite.
	
Let us use $*$ again to denote the Rosati involution on $\End(A)\otimes \QQ$ induced by $\lambda$. The algebra isomorphism $M_n(\End(A)\otimes \QQ))\cong \End(A^n)\otimes \QQ$ identifies the Rosati involution induced by $\lambda^n$ on $\End(A^n)\otimes\QQ$ with the positive involution $(X_{ij}) \rar (X^*_{ji})$ on $M_n(\End(A)\otimes \QQ)$. Thus the symmetric elements in $\End(A^n)\otimes \QQ$ are identified with hermitian matrices $H_n(\End(A)\otimes \QQ)$. The map $f \mapsto f_A$ is a morphism of formally real Jordan algebras $H_n(R_\QQ) \rar H_n(\End(A)\otimes \QQ)$. It is the same as the map induced by the $R$-action $\iota: R \hra \End(A)$, so it is in particular injective. Now considering $H_n(R_\QQ)$ as a subalgebra of $\End(A^n)\otimes\QQ$, it follows that $f_A > 0$ in $ \End(A^n)\otimes \RR$ if and only if $f>0$ in $H_n(R)\otimes \RR$. In other words, $f \otimes \lambda$ is a polarization if and only if $f$ is positive-definite, and so the same holds for $h$.
\end{proof}
	
Under some extra assumptions, we can extend this theorem further to characterize principal polarizations. Recall that a hermitian form $h: M \rar M^\vee$ is called non-degenerate if it's an isomorphism.
	
\begin{prop}\label{mainprop}
Under the conditions of Theorem \ref{mainthm}, suppose that furthermore $\lambda: A \rar A^\vee$ is principal, and that either:
\begin{itemize}
\item[(i)] The (left) $R$-module of endomorphisms $\End_S(A)$ is faithfully flat, 

\textit{or}
\item[(ii)] $R$ is commutative, and the $R$-module of $R$-linear endomorphisms $\End_R(A)$ is faithfully flat.	
\end{itemize}	
	
Then $h\otimes \lambda$ is a principal polarization if and only if $h$ is a non-degenerate positive definite hermitian form.
\end{prop}
	
\begin{proof}
It's clear that if $h$ is non-degenerate and $\lambda$ is principal, $h\otimes \lambda$ is an isomorphism. Conversely, suppose $h\otimes \lambda$ is an isomorphism. Since $h\otimes \lambda$ factors as a composition of the isomorphism $\mathbbm{1}_M \otimes \lambda: M\otimes_R A \rar M\otimes_R A^\vee$ with the map $h\otimes \mathbbm{1}_A : M\otimes_R A \rar M^\vee\otimes_R A$, the latter is also an isomorphism. Now consider this isomorphism on the $A$-valued points of $M\otimes_R A$:
\[ (h\otimes \mathbbm{1}_A)_A : M\otimes_R \End_S(A) \isomto M^\vee \otimes_R \End_S(A),\ \ \ m\otimes \phi \mapsto h(m)\otimes \phi. \]
Evidently, this is the tensor product of the map of right $R$-modules $h: M \rar M^\vee$ with the left $R$-module $\End_S(A)$. If the latter is faithfully flat, $h$ is an isomorphism.
	
Now suppose $R$ is commutative. Then $M$ is a bimodule, $M\otimes_R A$ and $M^\vee \otimes_R A^\vee$ inherit $R$-actions, and the isomorphisms $h\otimes \lambda$, $h \otimes \mathbbm{1}_A$ and $\mathbbm{1}_M \otimes \lambda$ are all $R$-linear. Consider the commutative diagram
	
\[ \begin{xymatrix}{ M\otimes_R \End_S(A) \ar[rr]^{(h \otimes \mathbbm{1}_A)_A} && M^\vee \otimes_R \End_S(A) \\
M\otimes_R \End_R(A) \ar@{^{(}->}[u] \ar[rr] && M^\vee \otimes_R \End_R(A) \ar@{^{(}->}[u]}.
\end{xymatrix} \]	
	
Knowing the top arrow is an isomorphism, we claim the bottom one is also one. Injectivity is clear from the diagram, so we must show surjectivity.
	
Let $\Psi: M^\vee \otimes_R A \rar M\otimes_R A$ be the inverse of $h\otimes \mathbbm{1}_A$. Since $h\otimes \mathbbm{1}_A$ is $R$-linear, so is $\Psi$. By Proposition \ref{homprop}(c) it corresponds to an element of $\Hom_R(M^\vee,M) \otimes_R \End_R(A)$, which is of the form $\sum_i \alpha_i \otimes s_i$, for $R$-linear $\alpha_i : M^\vee \rar M$ and $s_i \in \End_R(A)$. Then on $T$-valued points of $M^\vee \otimes_R A$, $\Psi$ is given by:
\[ \Psi_T : M^\vee \otimes_R A(T) \rar M \otimes_R A(T),\ \ \ f \otimes t \mapsto \sum_{i} \alpha_i(f) \otimes (s_i \circ t).\]
Letting $T=A$, we consider the map $\Psi_A$ restricted to $M^\vee \otimes_R \End_R(A) \subset M^\vee \otimes_R \End(A)$. If $t \in \End_R(A)$, we also have $s_i \circ t \in \End_R(A)$, and so $\Psi(f \otimes t)=\sum_i \alpha_i(f) \otimes (s_i \circ t) \in M\otimes_R \End_R(A)$. This shows the inverse of $(h\otimes \mathbbm{1}_A)_A$, when restricted to $M^\vee \otimes_R \End_R(A)$, takes values in $M\otimes_R \End_R(A)$. In other words $(h\otimes \mathbbm{1}_A)_A$ restricted to $M\otimes_R \End_R(A)$ is surjective onto $M^\vee \otimes_R \End_R(A)$. This proves the map 
\[ (h\otimes \mathbbm{1}_A)_A: M\otimes_R \End_R(A) \lra M^\vee \otimes_R \End_R(A),\ \ \ m\otimes s \mapsto h(m) \otimes s\]
is an isomorphism. Now, the above map is just the tensor product of $h$ with the identity map of the $R$-module $\End_R(A)$. Thus if the latter is faithfully flat, $h$ is an isomorphism.
\end{proof}

We note that in particular the proposition applies when $A$ is an abelian scheme with CM by the ring of integers $R=\OK$ of a CM field $K$, since then $\End_{R}(A)=R$.

\section{Tensor Product of Categories}

In this section we define the action of a monoidal category on another category, and the tensor product of two categories with such actions. Then we assume the monoidal category is a 2-group, which is to say its objects are invertible with respect to the monoidal product, and we show the morphisms of the tensor product in this case have a concise form. With the application to moduli spaces in mind, we assume one of the tensor factors is a groupoid and the other is fibred in groupoids over some base category. Then we show under certain conditions the resulting tensor product is also fibred in groupoids over the same base.

\subsection{Definitions}

A monoidal category is a category equipped with a product on objects resembling the tensor product of modules. We recall the definition from \cite[p. 162]{catmac}.

\begin{defn}  A \textbf{monoidal category} is a category $\calC$ equipped with the following data: a bifunctor $\Box: \calC \times \calC \rar \calC$, an identity object $e \in \calC$, and for all $a$, $b$, $c\in \calC$ a canonical \textit{associator} isomorphism
\[ \alpha_{a,b,c}: (a\square b)\square c \isomto a\square (b\square c),\]
along with left and right \textit{unitor} isomorphisms
\[ \lambda_a: e\square a \isomto a,\ \ \ \rho_a: a\square e \isomto a.\] 
The isomorphisms are required to satisfy the following \textit{pentagon} and \textit{triangle} relations:
\[
\begin{array}{ll}
\begin{xymatrix}{ & ((a\Box b)\Box c)\Box d \ar[dl]_{\alpha_{a,b,c}\Box d} \ar[dr]^{\alpha_{a \Box b,c,d}} & \\
(a \Box (b \Box c))\Box d \ar[d]_{\alpha_{a,b\Box c,d}} &  & (a \Box b)\Box (c \Box d) \ar[d]^{\alpha_{a,b,c\Box d}}\\
a \Box ((b \Box c)\Box d) \ar[rr]^{a \Box \alpha_{b,c,d}} & & a \Box (b \Box (c \Box d)) }\end{xymatrix} & 
\begin{xymatrix} {\\(a\Box e)\Box b \ar[rr]^{\alpha_{a,e,b}} \ar[dr]_{\rho_a \Box b} & & a\Box (e\Box b). \ar[dl]^{a\Box \lambda_b} \\ & a \Box b & } \end{xymatrix}
\end{array}
\] 
\end{defn}

We will omit the symbol $\Box$ and write $ab$ instead of $a\Box b$ for short. 

\begin{defn} A \textbf{left action} of a monoidal category $\calC$ on a category $\calX$ is the data consisting of: a bifunctor $\calC \times \calX \rar \calX: (a,X) \mapsto aX$, and for all $a$, $b \in C$, $X \in \calX$, canonical associator and left unitor isomorphisms
\[ \alpha_{a,b,X}: (ab)X \isomto a(bX),\ \ \ \ \lambda_X: eX \isomto X\] 
satisfying the pentagon and triangle relations
$$ \alpha_{a,b,cX} \circ \alpha_{ab,c,X} = a\alpha_{b,c,X} \circ \alpha_{a,bc,X} \circ \alpha_{a,b,c}X,\ \ \  a\lambda_X \circ \alpha_{a,e,X} = \rho_a X.$$

A \textbf{right action} of a monoidal category $\calC$ on a category $\calY$ is defined similarly, as a bifunctor $\calY \times \calC \rar \calY: (Y,a) \mapsto Ya$, with canonical associator isomorphisms $\alpha_{Y,a,b}: (Ya)b \isomto Y(ab)$, and right unitors $\rho_Y: Ye \isomto Y$, satisfying the analogous pentagon and triangle relations.
\end{defn}
By the coherence theorem of Mac Lane \cite[p.165] {catmac}, the pentagon and triangle relations are enough to ensure that all other expected associativity relations hold up to canonical isomorphisms. 

We want to a define a tensor product of categories over a monoidal category. Such tensor products are usually defined for additive categories \cite{SaavCat,DelCat}. For example, for $k$-linear categories where $k$ is a field, they have explicit constructions via generators and relations \cite{Tambara2001}. We require a similar notion, but for categories with no enriched structure. For this purpose we imitate the explicit construction in \cite{Tambara2001}, but leave out the additive features.

\begin{defn}\label{tensordef} Let $\calC$ be a monoidal category, and let 
$\calX$ (resp. $\calY$) be a category with a right (resp. left) action of 
$\calC$. The tensor product $\calX \otimes_\calC \calY$ is defined as the 
following category.	 The objects consist of symbols $X\otimes Y$, for $X\in 
\Ob(\calX)$ 
and $Y\in \Ob(\calY)$. The morphisms are words, modulo relations, in the 
following families of 
generator symbols:
\begin{itemize}
\item[I.] Symbols of the form $$\phi \otimes \psi: X\otimes Y \rar X'\otimes Y',$$
for all $(\phi: X \rar X') \in \Mor(\calX)$, $(\psi: Y \rar Y') \in \Mor(\calY)$.
\item[II.] Associators symbols 
\[ \alpha_{X,a,Y}: Xa\otimes Y \rar X\otimes aY,\] 
and their inverses
\[\alpha'_{X,a,Y}: X\otimes aY \rar Xa\otimes Y,\]
for all $X\in \calX$, $Y\in \calY$, $a\in \calC$. 
\end{itemize}
The relations imposed are:
\begin{enumerate}
\item[I.] Functoriality of $\otimes$:
\[ \mathbbm{1}_X \otimes \mathbbm{1}_Y = \mathbbm{1}_{X\otimes Y},\ \ \ (\phi \otimes \psi) \circ (\phi' \otimes \psi') = (\phi \circ \phi') \otimes (\psi \circ \psi'),\ \ {\phi \in \Mor(\calX),\ \psi \in \Mor(\calY)}.\]
\item[II.] Naturality of associators: commutativity of the diagram
\[ 
\begin{array}{ll}
\begin{xymatrix}{ Xa\otimes Y \ar[r]^{\alpha_{X,a,Y}} \ar[d]_{\phi u \otimes \psi} & X\otimes aY \ar[d]^{\phi\otimes u \psi}\\ 
X'a'\otimes Y' \ar[r]_{\alpha_{X',a',Y'}} & X'\otimes a'Y'.}\end{xymatrix} 
\end{array}
\]
for all $(\phi: X \rar X')\in \Mor(\calX), (\psi: Y \rar Y')\in \Mor(\calY)$, $(u: a \rar a')\in \Mor(\calC)$.
\item[III.] Isomorphic property of associators:
\[ \alpha_{X,a,Y} \circ \alpha'_{X,a,Y} = \mathbbm{1}_{X\otimes aY},\ \ \ \alpha'_{X,a,Y} \circ \alpha_{X,a,Y} = \mathbbm{1}_{Xa\otimes Y},\ \ \ \mathrm{for\ all\ }X\in \calX,\ Y\in \calY,\ a\in \calC.\]
\item[IV.] The pentagon and triangle relations:
\[ \alpha_{X,a,bY} \circ \alpha_{Xa,b,Y} = (\mathbbm{1}_X \otimes \alpha_{a,b,Y}) \circ \alpha_{X,ab,Y} \circ (\alpha_{X,a,b} \otimes \mathbbm{1}_Y),\ \ \ (\mathbbm{1}_X \otimes \lambda_Y) \circ \alpha_{X,e,Y} = \rho_X\otimes \mathbbm{1}_Y.\]
\end{enumerate}
\end{defn}

Again by Mac Lane's coherence theorem \cite[p. 165]{catmac} the pentagon and triangle relations above, together with their counterparts in the definitions of the monoidal category $C$, and the actions of $C$ on $\calX$ and $\calY$, imply that all expected associativity relations hold up to canonical isomorphism. Thus for instance, up to canonical isomorphism, the object $Xa_1a_2..a_r \otimes b_1b_2...b_s Y \in \calX\otimes_\calC \calY$ with $a_i$, $b_j \in \calC$ is well-defined.

\subsection{Tensor product over a 2-group}

A more concise representation of the morphisms just defined is possible when 
the objects in the monoidal category are invertible with respect to the 
monoidal 
product. This is the case for our application in $\S 3$.

\begin{defn}\label{2gpdef}
A \textbf{2-group} is a monoidal category $\calC$ such that for each $a\in \calC$ there exists another object $a^{-1}\in \calC$, and an isomorphism
\[ I_a: a\Box a^{-1} \isomto e,\]
where $e\in \calC$ is the identity object.
\end{defn}
Note that the object $a^{-1}$ is not necessarily unique, and neither is $I_a: a \Box a^{-1} \rar e$, even for a particular choice of $a^{-1}$. In the following, we will assume that the choices satisfy $(a^{-1})^{-1}=a$.

The 2-group we will later apply the results of this section to is 
$\Herm_1(\OK)$, the category of non-degenerate positive-definite rank-one 
hermitian modules over the ring of integers $\OK$ of a CM field $K$. The 
isomorphism classes of $\Herm_1(\OK)$ form the group of classes of hermitian 
forms, classically denoted $\frakC(K)$ \cite[$\S$14.5]{Shimura1998}.

Let $\calX$ and $\calY$ be categories with a right and left action by a 2-group $\calC$, respectively. To prevent the congestion of symbols later on, we define the following auxiliary isomorphisms. For $X \in \calX$, $Y\in \calY$, $a\in \calC$, we have
\begin{align} \mu_{X,a}: (Xa)a^{-1} \isomto X,\ \ \ \ \mu_{a,Y}: a^{-1}(aY) \isomto Y, \end{align}
given by $\mu_{X,a} = \rho_X \circ (\mathbbm{1}_X I_a)\circ 
\alpha_{X,a,a^{-1}}$, and $\mu_{a,Y} = \lambda_Y \circ (I_{a^{-1}} 
\mathbbm{1}_Y) \circ \alpha_{Y,a^{-1},a}^{-1}$. Note that these depend on a 
choice of $a^{-1}$, $I_a$ and $I_{a^{-1}}$. 

For each $X\otimes Y \in \calX \otimes \calY$, we also have an isomorphism 
\begin{align}\label{omegmap} \omega_{a,X,Y}: Xa \otimes a^{-1}Y \isomto X\otimes Y
\end{align}
given by 
\begin{equation} \label{formomeg} \omega_{a,X,Y} = (\mathbbm{1}_{X}\otimes \mu_{a^{-1},Y})\circ \alpha_{X,a,a^{-1}Y}. 
\end{equation}
In other words, $\omega_{a,X,Y}$ is the diagonal morphism in the diagram
\begin{align}\label{omegsquare} \begin{xymatrix}{
Xa\otimes a^{-1}Y \ar[d]_{\alpha_{Xa,a^{-1},Y}'} \ar[rr]^-{\alpha_{X,a,a^{-1}Y}} \ar[drr]^{\omega_{a,X,Y}} & & X \otimes a(a^{-1}Y) \ar[d]^{\mathbbm{1}_X \otimes \mu_{a^{-1},Y}}\\
(Xa)a^{-1} \otimes Y \ar[rr]_-{\mu_{X,a}\otimes \mathbbm{1}_Y} & & X\otimes Y,
}\end{xymatrix}\end{align}
which commutes as a consequence of relations II and IV in Definition \ref{tensordef}. 

For each $\phi \otimes \psi: X \otimes Y \rar X'\otimes Y'$ and $a\in\calC$, we also have a diagram
\begin{align}\label{omegfunc}\begin{xymatrix}{X a \otimes a^{-1}Y \ar[rr]^{\phi a \otimes a^{-1} \psi} \ar[d]_{\omega_{a,X,Y}}& & X'a \otimes a^{-1}Y' \ar[d]^{\omega_{X',a,Y'}}
\\ X \otimes Y \ar[rr]_{\phi\otimes \psi}& & X'\otimes Y',
}\end{xymatrix}
\end{align}
commuting as a consequence of relation II in Definition \ref{tensordef}, along with functorial properties of the action of $\calC$. 

We will show that every morphism $X'\otimes Y' \rar X\otimes Y$ in $\calX\otimes_\calC \calY$ can be written in the form $\omega_{a,X,Y}\circ (\phi\otimes \psi)$ for some $a\in \calC$, $\phi\in \Mor(\calY)$, $\psi\in \Mor(\calX)$. The following lemma is the essential reduction step in the proof.

\begin{lemma}\label{redlem} Suppose $\tau = \alpha_2 \circ (\phi\otimes \psi)\circ \alpha_1$ is a morphism of $\calX\otimes_\calC \calY$, where $\phi \in \Mor(\calX)$, $\psi\in \Mor(\calY)$ and $\alpha_1$, $\alpha_2$ are associator morphisms in $\calX\otimes_\calC \calY$. Then we can also write
\[ \tau =(\phi_1 \otimes \psi_1) \circ \alpha \circ (\phi_2 \otimes \psi_2),\]
where $\alpha$ is another associator, and $\phi_1$, $\phi_2 \in \Mor(\calX)$, $\psi_1$, $\psi_2 \in \Mor(\calY)$.
\end{lemma}

\begin{proof}
The associators $\alpha_1$ and $\alpha_2$ can each either have the form $\alpha_{X,a,Y}$ or its inverse $\alpha_{X,a,Y}'$. Of the four possibilities, we look at the case where $\alpha_1$ and $\alpha_2$ have the form
\[\alpha_1 = \alpha_{X,a,Y}: Xa\otimes Y \rar X\otimes aY, \ \ \ \alpha_2 =\alpha_{X',b,Y'}: X'b \otimes Y' \rar X'\otimes bY',\]
so that we have $\phi\otimes \psi: X\otimes aY \rar X'b\otimes Y'$. The other three cases are similar.

The claim then follows from the commutativity of the diagram
\[\begin{xymatrix}{ Xa\otimes Y \ar[rr]^{\alpha_{X,a,Y}} \ar[d]_{\mathbbm{1}_{Xa}\otimes \mu_{a,Y}^{-1}} && X\otimes aY \ar[rr]^{\phi\otimes \psi} && X'b \otimes Y' \ar[rr]^{\alpha_{X',b,Y'}}&& X'\otimes bY'\\
Xa\otimes a^{-1}aY \ar[rr]^{\phi a \otimes a^{-1}\psi} && (X'b)a \otimes {a}^{-1}Y' \ar[rr]^{\alpha_{X',b,a} \otimes \mathbbm{1}_{a^{-1}Y'}} && X'(ba)\otimes a^{-1}Y' \ar[rr]^{\alpha_{X',ba,a^{-1}Y'}} && X'\otimes (ba)(a^{-1}Y') \ar[u]^{\mathbbm{1}_{X'} \otimes F}.
}\end{xymatrix}
\]
Here the morphism $F$ is the composition 
\[ \begin{xymatrix}{ (ba)(a^{-1}Y') \ar[rr]^{\alpha'_{ba,a^{-1},Y'}} & & ((ba)a^{-1})Y' \ar[rr]^{(\alpha_{b,a,a^{-1}}) \mathbbm{1}_{Y'}} && (b(aa^{-1}))Y' \ar[rr]^-{(bI_a)\mathbbm{1}_{Y'}} && (be)Y'\ar[rr]^{\rho_b \mathbbm{1}_{Y'} }  && bY'}. \end{xymatrix}\]
Checking that this diagram does indeed commute is straight-forward using the axioms of $\calX \otimes_\calC \calY$. In particular, one uses the triangle and pentagon relations and the naturality of associators.
\end{proof}

Here is the main result on presentations of morphisms in $\calX\otimes_\calC \calY$.

\begin{prop}\label{morprop}
Let $\calX$ (resp. $\calY$) be categories with a right (resp. left) action of a 2-group $\calC$. Then every morphism $\tau: X\otimes Y \rar X'\otimes Y'$ in $\calX \otimes_\calC \calY$ has a presentation as a composition 
\begin{align} \begin{xymatrix}{
X\otimes Y \ar[rr]^-{\phi\otimes \psi} & &  X'a \otimes a^{-1}Y'\ar[rr]^-{\omega_{a,X',Y'}} & & X'\otimes Y',}
\end{xymatrix}
\end{align}
for some object $a\in \calC$, and morphisms $\phi$, $\psi$ in $\calX$, $\calY$, respectively. Alternatively, $\tau$ can also be written as $(\phi' \otimes \psi') \circ \omega_{a',X,Y}^{-1},$ for some other $\phi'$, $\psi'$, $a'$.
\end{prop}

\begin{proof}
By definition, a morphism $\tau$ of $\calX\otimes_\calC \calY$ is a string of symbols, each one of two types: associator morphisms $\alpha$, and tensor morphisms $\phi \otimes \psi$. Since a composition of two tensor morphisms is another tensor morphism, a word representing a general morphism can be reduced until the associators occurring in it are each separated by one tensor morphism (possibly the identity). Then as long as there remain at least two associators in the presentation of $\tau$, Lemma \ref{redlem} applies, and each time the number of associators can be reduced by one. The process necessarily ends with a presentation of the form $(\phi_1 \otimes \psi_1) \circ \alpha \circ (\phi_2 \otimes \psi_2$). 

Then to finish the proof it suffices to show the claim for a morphism of the form $(\phi \otimes \psi) \circ \alpha$. Assuming $\alpha = \alpha_{X,a,Y}$ and $\phi\otimes\psi: X\otimes aY \rar X' \otimes Y'$, this follows from the commutativity of the diagram
\[\begin{xymatrix}{ Xa \otimes Y \ar[rd]_{\mathbbm{1}_{Xa}\otimes \mu_{a,Y}^{-1}} \ar[r]^{\alpha_{X,a,Y}} & X\otimes aY \ar[rr]^{\phi \otimes \psi} & & X'\otimes Y'  \\
& Xa \otimes a^{-1}(aY) \ar[rr]^{\phi a \otimes a^{-1} \psi} \ar[u]_{\omega_{X,a^{-1},aY}}& & X'a \otimes a^{-1}Y'\ar[u]^{\omega_{X',a,a^{-1}Y}} \ar[rr]_{\alpha_{X',a,a^{-1}Y'}} & & X'\otimes a(a^{-1}Y') \ar[ull]_{\mathbbm{1}_{X'} \otimes \mu_{a^{-1},Y'}}}.
\end{xymatrix}
\]
The commutativity of the two triangles on the left and right follow from instances of (\ref{omegsquare}). The middle square is itself an instance of (\ref{omegfunc}).

In case $\alpha$ is of the form $\alpha'_{X,a,Y}$, a similar diagram gives a nearly identical presentation for $(\phi\otimes \psi) \circ \alpha$, wherein $a$ is replaced by $a^{-1}$. The alternative presentation of $\tau$ in the form $(\phi' \otimes \psi')\circ \omega_{a',X,Y}^{-1}$ results from yet other similar diagrams, with directions reversed.
\end{proof}

Now we consider the case where $\calX$ in $\calX \otimes_\calC \calY$ is a groupoid, and $\calY$ is fibred in groupoids over a base. We show that $\calX \otimes_\calC \calY$ is also fibred in groupoids, under some general conditions which we now define.

\begin{defn}\label{fibrewise} Let $p: \calY \rar \calS$ be a functor, and $\calC$ a monoidal category acting on $\calY$ on the left. Then $\calC$ is said to act \textbf{fibrewise} on $\calY$, if $p$ is a coequalizer in the diagram
\[ \begin{xymatrix}{ \calC\times \calY \ar@<0.5ex>[r]^{\Box} \ar@<-.5ex>[r]_{p_{\calY}} & \calY \ar[r]^p & \calS,
}\end{xymatrix}\\
\]
and if $p$ sends the associators and unitors of the action of $\calC$ to identity morphisms of $S$. Here $\Box$ denotes the action of $\calC$, and $p_{\calY}$ is projection onto the second factor.
\end{defn}

\begin{defn}\label{obfree}
The left action of a monoidal category $\calC$ on a category $\calY$ is said to be \textbf{free on objects} if whenever $aY \simeq bY$ in $\calY$ for some $Y\in \calY$, and $a$, $b \in \calC$, then $a\simeq b$ in $\calC$. 
\end{defn}

One can define a free action on the right analogously. These definitions appear essentially in \cite{Grayson1976} and \cite[pp. 339-340]{kbook}, though neither spell out the behaviour on associators and unitors for a fibrewise action. 

Recall that a category is called \textit{left-cancellative} if all its morphisms are monic. We introduce the following relative version.

\begin{defn} \label{relcan} A category $\calY$ lying over $\calS$ via $\pi: \calY \rar \calS$ is called \textbf{left-cancellative over} $\boldsymbol{\calS}$, if for any morphism $h: Z \rar X$ in $\calY$, and any pair of morphisms $f: X \rar Y$ and $g: X \rar Y$ such that $\pi(f)=\pi(g)$, we have $f=g$ whenever $h \circ f = h \circ g$.
\end{defn}

\begin{lemma}\label{canlem}A category fibred in groupoids $\calY \rar \calS$ is left-cancellative over $\calS$. 
\end{lemma}
\begin{proof}
Let $f: X \rar Y$, $g: X \rar Y$, and $h: Z \rar Y$ be morphisms in $\calY$, 
and suppose $f\circ h = g \circ h$, with $f$ and $g$ lying over the same 
morphism in $\calS$. Since $\calY$ is fibred in groupoids, every morphism of 
$\calY$ is cartesian. In particular $f$ is cartesian, hence there exists a 
unique morphism $\rho: X \rar X$ lying over $\mathbbm{1}_{S}$ such that $g = 
\rho \circ f$. Then we have $f \circ h = \rho \circ f \circ h$.  Since $f\circ 
h$ is also cartesian, $\rho$ is also the unique morphism $X \rar X$ over 
$\id_S$ satisfying $f\circ h = \rho \circ f \circ h$. As $\mathbbm{1}_X$ 
already satisfies this, $\rho = \mathbbm{1}_X$, therefore $g = f$.
\end{proof}

\begin{prop}\label{fiberedprop}
Let $\calX$ be a groupoid on which a 2-group $\calC$ acts on the right, and $p: \calY \rar \calS$ a category fibred in groupoids on which $\calC$ acts on the left fibrewise and free on objects. Suppose furthermore that $\calX \otimes_\calC \calY$ is left-cancellative over $\calS$. Then $\calX\otimes_\calC \calY$ is fibred in groupoids over $\calS$ via $\pi: \calX \otimes_\calC \calY \rar \calS$ defined by
\[ \pi(X\otimes Y) = p(Y),\ \ \ \pi(\phi \otimes \psi) = p(\psi),\ \ \ \pi(\alpha_{X,a,Y})=p(\mathbbm{1}_Y).\]
\end{prop}
\begin{proof}
Let $\alpha: T \rar S$ be a morphism in $\calS$, and $X\otimes Y$ an object in $\calX \otimes_\calC \calY$ over $S$. We must first show that $\alpha$ lifts to a morphism in $\calX \otimes_\calC \calY$ with target $X\otimes Y$. Now $Y\in \calY$ lies over $S\in \calS$. Since $\calY$ is fibred in groupoids over $\calS$, there exists an object $Y_T\in \calY$ over $T$ and a morphism $\psi: Y_T \rar Y$ lifting $\alpha$. Therefore $\mathbbm{1}_X \otimes \psi: X\otimes Y_T \rar X\otimes Y$ is a morphism lifting $\alpha$ to $\calX \otimes_\calC \calY$ . 

Now suppose $f: X' \otimes Y' \rar X \otimes Y$ lies over $\alpha: S' \rar S$, and $g: X''\otimes Y'' \rar X\otimes Y$ over $\beta: S'' \rar S$. Suppose $\gamma: S' \rar S''$ satisfies $\beta \circ \gamma = \alpha$. We must show there exists a unique morphism $h: X'\otimes Y' \rar X''\otimes Y''$ lying over $\gamma$, such that $g \circ h = f$.

Using the alternate presentation of a morphism given in Proposition \ref{morprop}, we can write $f = (\phi'\otimes \psi')\circ \omega_a$ where $\omega_a: X'\otimes Y' \rar X'a\otimes a^{-1}Y'$ is a canonical isomorphism for some $a \in \calC$, and similarly $g = (\phi'' \otimes \psi'') \circ \omega_b$ for some $b\in \calC$. Then $\pi(\omega_a) = \mathbbm{1}_{S'}$ and $\pi(\omega_b) = \mathbbm{1}_{S''}$ since $\calC$ acts fibrewise on $\calY$, and so $\phi' \otimes \psi'$ and $\phi'' \otimes \psi''$ are also lifts of $\alpha$ and $\beta$ to $\calX \otimes_\calC \calY$. Since $\omega_a$ and $\omega_b$ are isomorphisms, we have $g\circ h = f$ if and only if $(\phi''\otimes \psi'') \circ h' = (\phi'\otimes \psi')$, where $h'=\omega_b \circ h \circ \omega_a^{-1}$. Therefore it's enough to assume $f = \phi' \otimes \psi'$ and $g = \phi'' \otimes \psi''$, and show there's a unique $h$ lying over $\gamma$ such that $(\phi'' \otimes \psi'')\circ h = \phi' \otimes \psi'$.

We have $p(\psi')=\pi(\phi'\otimes \psi') = \alpha$, and $p(\psi'')=\pi(\phi''\otimes \psi'') = \beta$. As $\calY$ is fibred in groupoids over $\calS$, there exists a unique lift $\eta: Y' \rar Y''$ of $\gamma$ to $\calY$, such that $\psi'' \circ \eta = \psi'$. Now, the maps $\phi'$ and $\phi''$ are isomorphisms since $\calX$ is a groupoid. Hence, setting $\xi = {\phi''}^{-1} \circ \phi'$, we obtain a map $\xi \otimes \eta: X'\otimes Y' \rar X'' \otimes Y''$ lifting $\gamma$, which satisfies the desired property $(\phi'' \otimes \psi'') \circ (\xi \otimes \eta) = \phi' \otimes \psi'$. If $h: X'\otimes Y' \rar X'' \otimes Y''$ is any other lift such that $g \circ h = f = g \circ (\xi \otimes \eta)$ we have $h = \xi\otimes \eta$ since $\calX\otimes_\calC \calY$ is left-cancellative over $\calS$. This shows the lift we constructed is unique, which finishes the proof that $\calX\otimes_\calC \calY$ is fibred in groupoids over $\calS$.
\end{proof}

\section{Application to Moduli Spaces of Abelian Schemes.}

Using the results from \S 1, we apply the Serre tensor construction to certain 
moduli spaces of polarized abelian schemes related to PEL Shimura varieties. In 
the complex case, we show that we can construct all objects of the target 
moduli space in this way. Over a general base scheme, using deformation 
theory we show that all abelian schemes in the target family can be constructed 
\textit{\'etale locally on the base}. These results are formulated as an 
equivalence of categories in the complex case, and an isomorphism of stacks in 
general.

\subsection{The moduli space $\calM_\Phi^n$}

Let $K$ be a CM-field of degree $2g$ over $\QQ$, $\Phi$ a CM-type for $K$, and $n>0$ an integer. Let $L$ be the reflex field of $(K,\Phi)$. By $\OK$, resp. $\OL$, we denote the ring of integers of $K$, resp. $L$. We define a moduli space $\calM_\Phi^n$ over $\Spec \OL$ as follows. 

\begin{defn}\label{modspace}
For a locally noetherian scheme $S$ over $\Spec \OL$, $\calM_\Phi^n(S)$ is the category whose objects are triples $(A,\iota,\lambda)$ where:
\begin{itemize}
\item$A$ is an abelian scheme of relative dimension $ng$ over $S$.
\item $\iota: \OK \hra \End_S(A)$ is an injective ring homomorphism taking complex conjugation on $\OK$ to the Rosati involution on $\End_S(A)_\QQ$.
\item $\lambda: A \rar A^\vee$ is an $\OK$-linear principal polarization.
\end{itemize}
In addition, the triple $(A,\iota,\lambda)$ is required to satisfy the following \textit{ideal condition}. Let the ideal $J_\Phi$ be the kernel of the map
\begin{equation}\label{J} \OK\otimes \OL \rar \prod_{\phi\in \Phi} \CC^{(\phi)},\ \ \ (\alpha \otimes \beta)\mapsto (\phi(\alpha)\cdot \beta)_\phi.\end{equation}
Here $\CC^{(\phi)}$ denotes $\CC$, considered as a $K$-algebra via $\phi: K \hra \CC$. We require that the action of $\OK \otimes \OL$ on $\Lie_S(A)$ satisfy 
\begin{align} \label{idealcon} J_\Phi \Lie_S(A)=0.
\end{align}
The morphisms of $\calM_\Phi^n(S)$ are defined to be $\OK$-linear isomorphisms of abelian schemes preserving the polarizations.
\end{defn}

The functor $S \mapsto \calM_\Phi^n(S)$ defines a category fibred in groupoids over the category $\Sch_{/\OL}$ of locally noetherian $\OL$-schemes. It is representable by a Deligne-Mumford stack over $\Spec \OL$, which we also denote by $\calM_\Phi^n$. When $n=1$, it is an integral model of the stack of principally polarized abelian varieties with CM by $(K,\Phi)$. 

We first show $\calM_\Phi^n$ is \'etale and proper over $\Spec \OL$ (Theorem \ref{etprop}), generalizing results of B. Howard for the $n=1$ case \cite[Theorem 2.1.3]{unitary1}, as well as for $K$ quadratic imaginary \cite[Proposition 2.1.2]{unitary2}. Our proof is essentially the same, using the deformation theory of abelian schemes. The key to adapting Howard's proof is the ideal condition, which we now discuss.

In order to obtain a well-behaved moduli space, one typically imposes restrictions on the action of $\iota(a)$ induced on $\Lie_S(A)$, for all $a\in \OK$. For instance to obtain integral models of PEL Shimura varieties attached to unitary groups of a certain signature, one may impose a corresponding \textit{signature condition} by prescribing the characteristic polynomial of $\iota(a)$ acting on $\Lie_S(A)$. For us, the relevant abelian schemes are those that are, over $\CC$, isogenous to the $n$th power of a CM abelian variety of type $\Phi$. The corresponding signature condition is then
\begin{flalign} \label{charpo} \mathrm{charpoly}(\iota(a)|_{\Lie_S(A)},X) = \prod_{\phi\in \Phi} (X- \phi(a))^n,\end{flalign}
where the right hand side is identified with its image under the map $\OL[X] \rar \OS[X]$ induced by the structure morphism $S \rar \Spec \OL$.

However, over characteristic $p>0$ for $p$ ramified in $K$, the signature 
condition is not restrictive enough, since some embeddings $\phi\in \Phi$ may 
coincide. For example, let $K$ be quadratic imaginary, so that $L=K$ and $\Phi$ 
consists of a single embedding $\phi: K \hra \CC$. Let $S$ be a scheme over 
$\Spec \OK$ of characteristic $p$, where $p$ is ramified in $K$. Then for any 
$a\in \OK$, $a$ and $a^\sigma$ have the same image under $\OK \rar \OS$, so 
that $\phi$ and $\phi\sigma$ are indistinguishable using the $\OK$-action on 
$\OS$, and the signature condition is always satisfied. This is a general 
defect of the signature condition that causes the moduli space to acquire 
vertical components over ramified primes $p$, and so fail to be flat over 
$\Spec \OK$. 

For $K$ quadratic imaginary, the \textit{wedge condition} of G. Pappas 
\cite{Pappas2000}, formulated using exterior powers of $\iota(a)$ acting on 
$\Lie_S(A)$, is one approach to fixing the defect over ramified primes. The 
resulting moduli spaces are expected to be flat in general, and this has been 
verified in important special cases. The wedge condition corresponding to 
(\ref{charpo}) is simply $\iota(a) = \phi(a)$, i.e. that the two 
actions of 
$\OK$ on $\Lie_S(A)$ should coincide. The resulting moduli space is then proper 
and smooth over $\Spec \OK$, of relative dimension zero \cite[2.1.2]{unitary2}. 
However, it's not clear how to extend the wedge condition to the general CM 
case, since when $L\neq K$ there is no way to directly compare the actions of 
$\OK$ and $\OL$. 

For all CM fields $K$, the \textit{ideal condition} (\ref{idealcon}) fixes the 
above defect for the specific signature condition (\ref{charpo}), which we are 
interested in. It is equivalent to (\ref{charpo}) if $n=1$, or if $S$ has 
characteristic $0$. In general it \textit{implies} (\ref{charpo}) (see 
Corollary \ref{J->char}). 

If $A\in \calM_\Phi^1(S)$, so that it satisfies the ideal condition, and $M$ is 
a projective finitely presented $\OK$-module, then $M\otimes_\OK A$ also 
satisfies the ideal condition (by Lemma \ref{lielem}). Then if we want 
$\calM_\Phi^n$ to consist of objects arising from the Serre construction (at 
least \'etale locally), the ideal condition on $\calM_\Phi^n$ is 
\textit{necessary}. On the other hand, we would like $\calM_\Phi^n$ to have 
desirable properties such as flatness. We show that for this purpose the 
ideal condition is also \textit{sufficient}, in the sense that $\calM^n_\Phi$ 
as defined is \'etale and proper over $\Spec \OL$ (Theorem \ref{etprop}). This 
fact is key to the proof of the main theorem in the last section.

First we expose some basic properties of the ideal $J_\Phi$. Following \cite{unitary1}, let $\Lie_\Phi$ be defined by the exactness of the sequence of $\OK\otimes \OL$ modules
\begin{equation}\label{Jseq} \xymatrix{0 \ar[r] & J_\Phi \ar[r] & \OK\otimes \OL \ar[r] & \Lie_\Phi \ar[r] & 0}.\end{equation}
Since $\Lie_\Phi$ may be identified with the image of the map (\ref{J}), it is a projective $\OL$-module. Then the above sequence splits as 
$\OL$-modules. In particular $J_\Phi$ is a direct $\OL$-module summand of 
$\OK\otimes \OL$.

Recall that $\sigma$ denotes complex conjugation on $\OK$. We also use it to 
denote the induced $\OL$-linear automorphism on $\OK \otimes \OL$. 

\begin{lemma}\label{Jprops} The ideal $J_\Phi$ satisfies the following properties:
\begin{enumerate}
\item[(a)] $J_\Phi J_\Phi^\sigma = J_\Phi \cap J_\Phi^\sigma = 0$
\item[(b)] $J_\Phi$ is a projective $\OL$-module of rank $g$ (where $2g=[K:\QQ]$).
\item[(c)] Suppose $T$ is a local $\OL$-algebra, and $D$ is a free $(\OK\otimes 
T)$-module of rank $n$. Then $J_\Phi D$ is the unique direct summand of $D$, as 
a $T$-module, that is $\OK$-stable, has rank $ng$ over $T$, and satisfies 
$J_\Phi (D/M) = 0$.
\end{enumerate}
\end{lemma}
\begin{proof}

For an embedding $\phi: \OK \hra \CC$, let $\phi_L: \OK\otimes \OL \rar \CC$ denote the ring homomorphism $\alpha \otimes \beta \mapsto \phi(\alpha)\beta$. By definition, an element $x\in J_\Phi$ satisfies $\phi_{L}(x)=0$ for all $\phi\in \Phi$. Similarly, for $x\in J_\Phi^\sigma = J_{\Phi^\sigma}$ we have $\phi_{L}(x)=0$ for $\phi \in \Phi\sigma$. It follows that for $x\in J_\Phi \cap J_\Phi^\sigma$, we have $\phi_{L}(x)=0$ for all embeddings $\phi: \OK \hra \CC$. But the map
$$\OK \otimes \OL \rar \prod_{\phi: \OK \hra \CC} \CC^{(\phi)},\ \ \ (\alpha \otimes \beta) \mapsto (\phi(\alpha)\beta)_\phi$$ 
is injective, so $x=0$, which shows $J_\Phi \cap J_\Phi^\sigma = 0$. Since $J_\Phi J_\Phi^\sigma \subseteq J_\Phi \cap J_\Phi^\sigma$, this proves (a).

For (b), we first note that $\OK \otimes \OL$ is free of rank $2g$ over $\OL$. 
Since $J_\Phi$ is an $\OL$-submodule of $\OK\otimes \OL$, it is torsion-free, 
and hence projective. The rank can be verified over $\CC$ by 	applying 
$\otimes_{\OL} 
\CC$ to (\ref{Jseq}), which becomes
$$ 0 \lra \prod\limits_{\phi \in \Phi\sigma} \CC^{(\phi)} 
\lra\prod_{\phi: K\hra \CC} \CC^{(\phi)} \lra \prod_{\phi \in \Phi} 
	\CC^{(\phi)}\lra 0.$$

For part (c), recall that $J_\Phi$ is a direct summand of $\OK\otimes \OL$ as 
an $\OL$-module, since as such the exact sequence (\ref{Jseq}) is split. It 
follows that $J_\Phi D$ is a direct summand of $D$ as a $T$-module. As it is 
isomorphic to $J_\Phi (\OK\otimes T)^n \cong (J_\Phi \otimes_\OL T)^n$, by (b) 
it also has rank $ng$ over $T$, so it satisfies the properties mentioned in 
part (c). Now suppose $M$ is another such direct summand satisfying these 
properties. The condition 
$J_\Phi (D/M)=0$ implies $J_\Phi D \subset M$. It's an exercise in commutative 
algebra to show that in a free-module of finite rank, if one direct summand is 
contained in another one, and both have the same rank, they are equal. Hence 
$M=J_\Phi D$, showing uniqueness.
\end{proof}

For an abelian scheme $A$ defined over $S$, we denote the first algebraic de Rham homology $H_1^\DR(A/S) = \calH om_{\OS}(H^1_\DR(A/S), \OS)$ by $\DD_A(S)$. There's a fundamental Hodge filtration, an exact sequence of locally free $\OS$-modules 
\begin{align} \label{H} 0 \rar \Fil^1 \DD_A(S) \rar \DD_A(S) \rar \Lie_S(A) \rar 0. \end{align}

When $S=\Spec T$, the above may be identified with an exact sequence of projective $T$-modules by passing to global sections. We will often make this identification when $S$ is affine.

\begin{prop} \label{Dfree} Let $T$ be a local $\OL$-algebra with a separable residue field $\FF$, $S=\Spec T$, and $(A,\iota,\lambda)\in \calM_\Phi^n(S)$. Then:
\begin{enumerate}
\item[(a)] $\DD_A(S)$ is free of rank $n$ over $\OK\otimes T$.
\item[(b)] The choice of an isomorphism $(\OK\otimes T)^n \isomto \DD_A(S)$ leads to an isomorphism of short exact sequences  
\begin{equation*}\xymatrix{0 \ar[r] & J_\Phi(\OK \otimes T)^n \ar[r] \ar[d]^\simeq & (\OK \otimes T)^n \ar[r]\ar[d]^\simeq & \Lie_\Phi \otimes_\OL T^n \ar[r]\ar[d]^\simeq & 0\\
0 \ar[r] & \Fil^1 \DD_A(S) \ar[r] & \DD_A(S) \ar[r] & \Lie_S(A) \ar[r] & 0. }
\end{equation*}
\end{enumerate}
\end{prop}
\begin{proof}
For (a), we first consider the case $T= \FF$. If the characteristic of $\FF$ is zero, then $\DD_A(F)=H_1^\DR(A/\FF)$ is free of rank $n$ over $\OK\otimes \FF$ by comparison with Betti homology. If the characteristic is $p>0$, one first shows that the covariant Dieudonn\'e module $D(A)$ is free of rank $n$ over $\OK\otimes W(\FF)$. This is proved in \cite[Lemme 1.3]{Rapo78}, where it is stated in terms of $H^1_{cris}(A)$. The result then follows by $H^{\DR}_1(A/\FF) \cong D(A)\otimes_{W(\FF)} \FF$. 

Now let $T$ be any local ring with residue field $\FF$ and $S=\Spec T$. Let $A_0$ denote $A\otimes \FF$. We have $\DD_A(S)\otimes_T \FF \cong \DD_{A_0}(\FF) \simeq (\OK\otimes \FF)^n$. Let $\{x_1,\cdots,x_n\}$ be the lift to $\DD_A(S)$ of an $(\OK\otimes \FF)$-basis for $\DD_{A_0}(\FF)$, and let $(\OK\otimes T)^n \rar \DD_A(S)$ be the $\OK\otimes T$-linear map sending $e_i$ to $x_i$. By Nakayama's lemma for the local ring $T$, this map is surjective. Let $K$ denote the kernel. Since $\DD_A(S)$ is projective over $T$, we have $(\OK\otimes T)^n \simeq K\oplus \DD_A(S)$, which shows $K$ is also projective, hence free. Now applying $\dash \otimes_T \FF$ to the isomorphism $(\OK\otimes T)^n \simeq K\oplus \DD_A(S)$ shows that $K\otimes_T \FF=0$, which implies $K=0$ by considering rank. Thus the map $(\OK \otimes T)^n \rar \DD_A(S)$ is also injective, hence an isomorphism.

For part (b), note that since $(\OK\otimes T)^n \cong (\OK\otimes \OL)\otimes_\OL T^n$, the first row can be obtained by tensoring (\ref{Jseq}) with $T^n$ over $\OL$, so it is exact. The ideal condition $J_\Phi \Lie_S(A) = 0$ and the exactness of the second row together imply that the composition $J_\Phi(\OK \otimes T)^n \rar (\OK\otimes T)^n \rar \DD_A(S)$ lands in $\Fil^1 \DD_A(S)$, providing the map on the left. Exactness of the first row then provides the map on the right. As $\Lie_S(A)$ is a projective $T$-module of dimension $ng$, $\Fil^1 \DD_A(S)$ is a direct summand of $\DD_A(S)$ satisfying the conditions in Lemma \ref{Jprops}(c), so it must coincide with $J_\Phi \DD_A(S)$, which is the image of $J_\Phi (\OK \otimes T)^n$ in $\DD_A(S)$. Therefore the map on the left is also an isomorphism. Then since the vertical maps in the middle and the left are isomorphisms, so is the one on the right.
\end{proof}

\begin{cor}\label{J->char}Let $S$ be a scheme locally of finite type over $\Spec \OL$, and $(A,\iota,\lambda)\in \calM_\Phi^n(S)$. Then for each $a\in \OK$, 
$$\mathrm{charpoly}(\iota(a)|_{\Lie_S(A)},X) = \prod_{\phi\in \Phi} (X- \phi(a))^n.$$
\end{cor}
\begin{proof}
The assertion is an identity of global sections of $\OS[X]$, the given 
polynomial being identified with its image under $\OL[X] \rar \OS[X]$. Since 
such an identity may be checked at the stalks of $\OS[X]$, we may assume 
$S=\Spec T$, where $T$ is a local $\OL$-algebra. Furthermore, since $S$ is 
locally of finite type over $\Spec \OL$, the residue field of $T$ either has 
characteristic zero, or is a finite extension of the residue field of a closed 
point in $\Spec \OL$. In either case, it is separable. Now by Lemma 
\ref{Dfree}(b), $\Lie_S(A)$ is isomorphic to $\Lie_\Phi \otimes_\OL T^n$ as an 
$\OK\otimes T$-module. The characteristic polynomial of $a \in \OK$ acting on 
$\Lie_\Phi$ can be seen to equal $\prod_{\phi\in \Phi}(X-\phi(a))$ after 
extending scalars to $\CC$ and using the 
fact $\Lie_\Phi \otimes_{\OL}\CC \simeq \prod_{\phi \in \Phi} 
\CC^{(\phi)}$. The image of the same polynomial under $\OL[X] \rar T[X]$ gives 
the characteristic polynomial of $a\in \OK$ acting on $\Lie_\Phi \otimes_\OL 
T$. Therefore $a\in \OK$ acting on $\Lie_S(A) \simeq \Lie \otimes_\OL T^n \cong 
(\Lie \otimes_\OL T)^n$ has characteristic polynomial $\prod_{\phi \in \Phi} 
(X-\phi(a))^n.$
\end{proof}

We now turn to the deformation theory of abelian schemes to prove 
$\calM_\Phi^n$ has the expected properties. Let us fix a point $y\in \Spec 
\OL$. The completed 	\'etale local ring $\wh{\calO_{L,y}^\sh}$ of $\Spec 
\OL$ at 
$y$ is the completion of the ring of integers of the maximal unramified 
extension of $\calO_{L,y}$. Following the notation of \cite{unitary1}, we 
denote it by $W_\Phi$. Its residue field, a separable closure of $k(y)$, will 
be denoted $\FF$. Let $\calC_\Phi$ be the category of complete local noetherian 
$W_\Phi$-algebras with residue field $\FF$, and $\calA_\Phi$ its subcategory of 
artinian rings. The relevant facts from deformation theory of abelian schemes 
\cite[Ch. 2]{Lansthesis} are summarized as follows. 

Let $T' \thrar T$ be a surjection in $\calC_\Phi$ with kernel $I$ satisfying $I^2=0$. Set $S=\Spec T$ and $S'=\Spec T'$. An abelian scheme $A/S$ always lifts to some abelian scheme $A'/S'$, and for any such $A'$ there's a canonical isomorphism $\DD_{A'}(S')\otimes_{T'} T \cong \DD_A(S)$. Furthermore, the $T'$-module $\DD_{A'}(S')$ is up to canonical isomorphism independent of the choice of the lift. For convenience we hide away the canonical isomorphisms, erasing $A'$ from the notation and writing $\wt{\DD}_A(S')$ instead of $\DD_{A'}(S')$. The submodule $\Fil^1 \DD_{A'}(S')$ of $\wt{\DD}_A(S')$ on the other hand does depend on the choice of the lift, and determines it completely. More specifically, there is a bijection between lifts $A'/S'$ of $A/S$, and projective $T'$-submodules $M$ of $\wt{\DD}_A(S')$ such that $M\otimes_{T'} T \cong \Fil^1 \DD_{A}(S)$ via the isomorphism $\wt{\DD}_{A}(S') \otimes_{T'} T \cong \DD_A(S)$.

Let $A'/S'$ be a lift of $A/S$. For an element $\phi \in \End_S(A)$, the induced $T$-module endomorphism of $\DD_A(S)$ lifts canonically to a morphism $\phi_*: \wt{\DD}_A(S') \rar \wt{\DD}_A(S')$. Then $\phi$ lifts (uniquely) to an endomorphism of abelian schemes $\phi': A' \rar A'$ if and only if $\phi_*$ leaves the corresponding $T'$-submodule $\Fil^1\DD_{A'}(S')$ invariant. In particular, an $\OK$-action $\iota: \OK \hra \End_S(A)$ lifts (uniquely) to an $\OK$-action $\iota'$ on $A'$ if and only if $\Fil^1\DD_{A'}(S')$ is an $\OK$-submodule of $\wt{\DD}_A(S')$.

A polarization $\lambda: A \rar A^\vee$ induces a symplectic pairing $\langle\ ,\ \rangle_\lambda$ on $\DD_A(S)$ which lifts to a pairing on $\wt{\DD}_A(S')$, denoted $\langle\ ,\ \rangle_\lambda'$. Given a lift $A'/S'$ of $A/S$, $\lambda$ lifts (uniquely) to a map $\lambda': A' \rar {A'}^\vee$ if and only if the submodule $\Fil^1 \DD_{S'}(A')$ of $\wt{\DD}_A(S')$ is totally isotropic for $\langle\ ,\ \rangle_\lambda'$. In that case $\lambda'$ is a polarization for $A'$, and it is principal if $\lambda$ is. If $A$ has an $\OK$-action that lifts to $A'$ and $\lambda$ is $\OK$-linear, so is $\lambda'$. In addition, the Rosati involution induced by $\lambda$ corresponds to the adjoint for the pairing $\langle\ ,\ \rangle_\lambda'$, so that $\langle ax,y \rangle_\lambda' = \langle x,a^\sigma y \rangle_\lambda'$ for all $x,y\in \wt{\DD}_A(S')$, $a\in \OK$.

Putting the above facts together, we see that lifting a triple $(A,\iota,\lambda)$ over $S$ to $(A',\iota',\lambda')$ over $S'$ is equivalent to lifting the Hodge filtration $\Fil^1 \DD_A(S)$ of $A$ to an $\OK$-stable projective $T'$-submodule of $\wt{\DD}_A(S')$ that is totally isotropic with respect to the pairing $\langle\ ,\ \rangle_\lambda'$. Such a lift satisfies the ideal condition if and only if $J_\Phi \wt{\DD}_A(S') \subseteq \Fil^1 \DD_{A'}(S')$.

\begin{prop}\label{1lift} Every object in $\calM_\Phi^n(\FF)$ lifts uniquely to $\calM_\Phi^n(T)$, for all $T\in \calC_\Phi$.
\end{prop}
\begin{proof}Every $T$ in $\calC_\Phi$ is an inverse limit of its quotients, which are artinian rings in $\calA_\Phi \subset \calC_\Phi$. Then it's enough to show the claim for $T\in \calA_\Phi$, since $\calM_\Phi^n$ is an algebraic stack for which formal deformations are effective. Each such artinian ring has a surjective map to $\FF$ which is a composition of finitely many surjections in $\calA_\Phi$, with square-zero kernels. Therefore it suffices to show that for a surjective map $T' \thrar T$ in $\calA_\Phi$ having square-zero kernel, with $S = \Spec T$ and $S' = \Spec T'$, every object $(A,\iota,\lambda)$ of $\calM_\Phi^n(S)$ lifts uniquely to an object of $\calM_\Phi^n(S')$. For such an object $(A,\iota,\lambda)$, we have the Hodge filtration of projective $T$-modules (\ref{H}). By Proposition \ref{Dfree}, $\Fil^1 \DD_A(S)$ is the $T$-submodule $J_\Phi \DD_A(S)$ of $\DD_A(S)$. For the same reason, any lift $(A',\iota',\lambda')$ of $(A,\iota,\lambda)$ to $S'$, if such exists, would have $\Fil^1 \DD_{A'}(S')= J_\Phi \DD_{A'}(S') = J_\Phi \wt{D}_A(S')$, and would therefore be unique, by the deformation theory outlined above. 

Now we claim $M=J_\Phi \wt{\DD}_A(S')\subset \wt{\DD}_A(S')$ does indeed correspond to a lift of $(A,\iota,\lambda)$ to $S'$. It lifts the Hodge filtration since $M\otimes_{T'} T = J_\Phi \wt{\DD}_A(S')\otimes_{T'} T \cong J_\Phi \DD_A(S)$, which is equal to $\Fil^1 \DD_A(S)$ as we have just noted. Since $M$ is $\OK$-stable, the pair $(A,\iota)$ lifts (uniquely) to a pair $(A',\iota')$. Since $J_\Phi \wt{\DD}_A(S') \subseteq M$, the pair $(A',\iota')$ satisfies the ideal condition. It remains to show that $M$ is totally isotropic for $\langle\ ,\ \rangle_\lambda'$. For $x,y\in \wt{\DD}_A(S')$ and $r,s\in J_\Phi$, we have
$$ \langle rx,sy \rangle_\lambda' = \langle s^\sigma r x,y \rangle_\lambda' = 0,$$
since $s^\sigma r =0$ by Lemma \ref{Jprops}(a).
\end{proof}
\begin{thm}\label{etprop}
The stack $\calM_\Phi^n$ is \'etale and proper over $\Spec \OL$.
\end{thm}
\begin{proof}
Let $\FF$ be a separably closed field, and $\overline{x}: \Spec \FF \rar \calM_\Phi^n$ a geometric point of $\calM_\Phi^n$. Let $\overline{y}: \Spec \FF \rar \Spec \OL$ be the underlying geometric point, with image $y\in \Spec \OL$. Fix a surjective \'etale morphism $M \thrar \calM$ from a scheme $M$. Then $\ol{x}$ lifts to a map $\Spec \FF \rar M$ with image $x\in M$ lying over $y\in \Spec(\OL)$, and $M \rar \Spec(\OL)$ induces a map of \'etale local rings $\calO_{L,y}^\sh \rar \calO_{M,x}^\sh$. To show $\calM_\Phi^n \rar \Spec(\OL)$ is \'etale, it's enough to show the induced map on completions $$\wh{\calO_{L,y}^\sh} \rar \wh{\calO_{M,x}^\sh}$$ is an isomorphism. The ring $\wh{\calO_{L,y}^\sh}$, which is the completion of the ring of integers of the maximal unramified extension of $\calO_{L,y}$, will be denoted $W_\Phi$ as before. The ring $\wh{\calO_{M,x}^\sh}$, which is up to isomorphism independent of the choice of $M\thrar \calM_\Phi^n$, will be denoted $R_\calM$. 

The geometric point $\overline{x}$ lifts uniquely to a map $\Spec \FF \rar \Spec \RM$, whose image lies over $x\in M$. The image of the corresponding ring homomorphism $\RM \rar \FF$ is the separable closure $\FF'$ of the residue field $k(x)$ of $\calO_{M,x}$ in $\FF$. Then $\Spec \FF \rar \Spec \RM$ factors uniquely through $\Spec \FF \rar \Spec \FF'$, so it is harmless to assume $\FF$ itself is the residue field of $\RM$. Then since $M \rar \Spec \OL$ is locally of finite type, $\FF$ is also the residue field of $W_\Phi$.

Let $\calC_\Phi$ be the category of complete local noetherian $W_\Phi$-algebras with residue field $\FF$. If $(A_x,\iota_x,\lambda_x)\in \calM_\Phi^n(\FF)$ corresponds to $\overline{x}: \Spec \FF \rar \calM_\Phi^n$, then for any $T\in \calC_\Phi$, the set $\Hom_{W_\Phi}(\RM,T)$ corresponds to lifts of $(A_x,\iota_x,\lambda_x)$ to $\calM_\Phi^n(T)$. By Proposition \ref{1lift}, there is a unique such lift for every $T$, hence a unique morphism $\RM \rar T$ of $W_\Phi$-algebras. In other words, $\RM$ is an initial object of $\calC_\Phi$. Since $W_\Phi$ is also an initial object, the map $W_\Phi \rar \RM$ must be an isomorphism. This proves $\calM_\Phi^n \rar \Spec \OL$ is \'etale.

The fact that $\calM_\Phi^n$ is proper over $\Spec \OL$ follows from the 
valuative criterion of properness. The proof is identical to the quadratic 
imaginary case in \cite[Proposition 2.1.2]{unitary1}. 
\end{proof}

Next we begin the systematic construction of the morphisms and objects of $\calM_\Phi^n$.

\subsection{Serre construction of $\calM_\Phi^n$: the morphisms}

Let $(K,\Phi)$ be, as in the previous section, a CM-field of degree $2g$ over $\QQ$, and set $R=\OK$. Let $\Herm_n(R)$ denote the category of pairs $(M,h)$ consisting of a projective finitely presented $R$-module $M$ of rank $n$, and a non-degenerate positive-definite $R$-hermitian form $h: M \rar M^\vee$. It follows from Theorem \ref{mainthm} that given an object $(A,\iota,\lambda)$ of $\calM_\Phi^1(S)$, and another $(M,h)$ of $\Herm_n(R)$, the triple
$$(M\otimes_R A, \iota \otimes \mathbbm{1}_A, h\otimes \lambda)$$
is a well-defined object of $\calM_\Phi^n(S)$. We will denote it by
$$ (M,h) \otimes (A,\iota,\lambda).$$

It follows that the 2-group $\Herm_1(R)$ acts on $\calM_\Phi^1(S)$ on the left via the Serre construction. It also acts on $\Herm_n(R)$ on the right via ordinary tensor product. Thus, as described in $\S 2$, we can form the the tensor product category
$$ \Herm_n(R) \otimes_{\Herm_1(R)} \calM_\Phi^1(S).$$

To avoid confusion with the Serre construction, we denote the objects of this category by
$$(M,h) \boxtimes (A,\iota,\lambda).$$
Likewise, we denote the pure tensor morphisms in the same category by $f\boxtimes \phi$. There is then a functor 
$$\Sigma_S: \Herm_n(R) \otimes_{\Herm_1(R)} \calM_\Phi^1(S) \lra \calM_\Phi^n(S),$$
given on objects by
$$ \Sigma_S: (M,h)\boxtimes (A,\iota,\lambda) \mapsto (M,h) \otimes (A,\iota,\lambda),$$
sending morphisms $f\boxtimes \phi$ to $f\otimes \phi$, and mapping associator isomorphisms to their counterparts. The purpose of most of this section is to prove the following.
\begin{prop}\label{fullfaith}
The functor $\Sigma_S: \Herm_n(R)\otimes_{\Herm_1(R)}\calM_\Phi^1(S) \lra \calM_\Phi^n(S)$ is fully faithful.
\end{prop}

We will first characterize the morphisms in $\calM_\Phi^n$ that are in the image of the functor $\Sigma_S$, and then compare our result with the description given in $\S$2 of the morphisms of $\Herm_n(R)\otimes_{\Herm_1(R)}\calM_\Phi^1(S)$. 

For general abelian schemes $A$ and $B$ over a connected base $S$, and a point $s\in S$, by the rigidity lemma of \cite[Ch. 6]{GIT}, the map $\Hom_S(A,B) \rar \Hom_{k(s)}(A_s,B_s)$ given by base change to the fibre at $s$ is injective. For $R$-linear maps of abelian schemes in $\calM_\Phi^1(S)$, we have the following.

\begin{thm}\label{hombij}Let $(A,\iota,\lambda)$, $(B,\jmath,\mu) \in 
\calM_\Phi^1(S)$ 
with $\Hom_R(A,B)\neq 0$. Then for any morphism $T \rar S$ of connected locally 
noetherian $\OL$-schemes, the map $\Hom_R(A,B) \rar \Hom_R(A_T,B_T)$ induced by 
base change is a bijection.
\end{thm}
\begin{proof}
By the rigidity lemma of \cite[\S 6.1]{GIT}, the map in question is injective, 
so it's enough to show surjection.

First we show the hom sheaf $\calH=\ul{\Hom}_R(A,B)$ is representable by an 
\'etale scheme over $S$. It is well-known that $\calH$ is representable by a 
scheme locally of finite type over $S$ (e.g. \cite[\S6]{Hpadic}, or
\cite[\text{proof of }1.4.4.5]{cmlift}). $\calH/S$ is locally of 
finite presentation, since it's locally of finite type and $S$ is locally 
noetherian. Then to show $\calH$ is \'etale over $S$ it's enough to show it's 
formally \'etale. 

Let $T_0\hra T$ be a closed immersion of 
$S$-schemes defined by a square-zero sheaf of $\calO_T$-ideals. Let $u: 
A_{T_0} \rar B_{T_0}$ be in $\calH(T_0)$. We must show $u$ lifts to 
an $R$-linear morphism 
$\wt{u}: A_T \rar B_T$. Uniqueness is automatic by infinitesimal rigidity 
\cite[Lemma 2.2.2.1]{Lansthesis}. For any such $u$, 
the induced map $u_*: \DD_A(T_0) \rar \DD_B(T_0)$ on de Rham homology lifts 
to a $\wt{u}_*: \wt{\DD}_A(T) \rar \wt{\DD}_B(T)$ of 
$\calO_T$-modules \cite[2.1.6.4]{Lansthesis}. The morphism $u$ lifts to a 
$\wt{u}: A_T \rar B_T$ if and only if $\wt{u}_*$ respects the Hodge filtrations 
\cite[2.1.6.9]{Lansthesis}. By Proposition \ref{Dfree}, we must show 
$\wt{u}_*(J_\Phi \wt{\DD}_A(T))\subset J_\Phi\wt{\DD}_B(T)$. This follows from 
$J_\Phi \subset R\otimes\OL$, and the fact that $\wt{u}_*$ is $R \otimes 
\OL$-linear. Thus $\calH \rar S$ is formally \'etale, hence \'etale. It 
is also surjective since $\calH$ as an $S$-group scheme admits a 
section.

The valuative criterion of properness implies, as in \cite[Corollary 
6.9]{Hpadic}, that every connected component of $\calH$ is proper, and hence 
finite \'etale, over $S$. 

Let $s \rar S$ be a fixed geometric point. Applying the equivalence of 
\cite[Theorem I.5.3]{Milne_Etale} to every connected component of $\calH$ 
equips $\calH_s$ with an $R$-linear action of $\pi_1^{\mathrm{et}}(S,s)$. The 
fundamental group acts trivially on the image of $\Hom(A,B) \rar \calH_s$, 
which has finite index since $\Hom(A,B)$ and $\calH_s$ are rank-one projective 
$R$-modules. It follows that the action is trivial. Then by \cite[Theorem 
I.5.3]{Milne_Etale}, each connected component of $\calH$ is a constant 
$S$-scheme, isomorphic to $S$ itself. 
	
Hence $\calH$ is a disjoint union of copies of $S$, from which the theorem 
follows.
\end{proof}

\bigskip
Let $S\in \Sch_{/\calO_L}$ be connected, and suppose $(A,\iota,\lambda)$, 
$(B,\jmath,\mu)$ are in $\calM_\Phi^1(S)$. Any $f\in \Hom_R(A,B)$ has a 
\textit{Rosati dual} $f'\in \Hom_R(B,A)$ defined by $f'=\lambda^{-1}\circ 
f^\vee \circ \mu$. 

\begin{prop}\label{herm} Let $(A,\iota,\lambda)$, $(B,\jmath,\mu)$ $\in 
\calM_\Phi^1(S)$, 
with $\Hom_R(A,B)\neq 0$. Let $\Hom_R(B,A)\otimes_R B$ be equipped with the 
induced $R$-action. Then:
\begin{itemize}
\item[(a)] The canonical map $\Hom_R(B,A) \otimes_R B \rar A$ is an $R$-linear 
isomorphism.
\item[(b)] The map 
$$\Hom_R(B,A)\otimes_R \Hom_R(A,B) \rar \End(A),\ \ f\otimes g 
\mapsto f\circ g$$ is an isomorphism of $R$-modules.
\item[(c)] The map
$$ \Hom_R(A,B)\times \Hom_R(A,B) \rar R,\ \ \ (f,g) \mapsto \iota^{-1}(g'\circ 
f)$$
is a positive-definite non-degenerate hermitian form.
\end{itemize}
\end{prop}
\begin{proof}
The map in (a) is an isogeny, so its kernel is a finite group scheme. 
Its triviality can then be checked on geometric points $\Spec(k) \rar S$. By
Proposition \ref{1lift} one can reduce to characteristic zero, and even assume 
$k=\CC$. In that case we have $B(\CC)\simeq \CC^g/\Phi(\frakb)$, where $\frakb$ 
is a fractional ideal of $K$ and $\Phi(\frakb)$ is the image of
$$ \frakb \rar \CC^g,\ \ \ b \mapsto (\phi_1(b),\phi_2(b),\cdots,\phi_g(b)),$$
with $(\phi_1,\cdots,\phi_g)$ a fixed ordering of $\Phi$. Similarly 
$A(\CC)\simeq \CC^g/\Phi(\fraka)$ for a fractional ideal $\fraka$, and so 
$\Hom_R(B,A)\cong \frakb^{-1}\fraka$. The canonical map $\Hom_R(B,A)\otimes_R B 
\rar A$ is then induced by the multiplication map 
$\frakb^{-1}\fraka\otimes_R \frakb \rar \fraka,$ which is clearly an 
isomorphism. 
This proves (a).

For (b), the given map is the composition of the following canonical 
isomorphisms
\begin{align*} \Hom_R(B,A)\otimes_R \Hom_R(A,B) &\cong 
\Hom_R(R,\Hom_R(B,A))\otimes_R \Hom_R(A,B)\\
&\cong \Hom_R(R\otimes_R A, \Hom_R(B,A)\otimes_R B)  & &\text{(by Prop. 
\ref{homprop}(c))}\\
&\cong \Hom_R(A,A) = \End_R(A)&& \text{(using (a))}
\end{align*}

It's clear that the given pairing in (c) is hermitian. For non-zero $f\in 
\Hom_R(A,B)$, 
we 
have $\lambda\circ (f'\circ f) = 
f^\vee \circ \mu \circ f$ which is a polarization on $A$, so 
$\iota^{-1}(f'\circ f)$ is totally positive by Corollary \ref{ampcor}. The 
bilinear map corresponds to the linear map in $(b)$ up to composition with the Rosati dual and 
the isomorphism $\iota$. Its non-degeneracy is then equivalent to the 
isomorphism in (b).
\end{proof}

The following lemma is a step towards characterizing the morphisms of 
$\calM_\Phi^n(S)$.

\begin{lemma}\label{pureten} Let $(M,h)$, $(N,k) \in \Herm_n(R)$, and 
$(A,\iota,\lambda)$, $(B,\jmath,\mu) \in \calM_\Phi^1(S)$. Suppose that $f: M 
\rar N$ and $\phi: A \rar B$ are $R$-linear homomorphisms, such that $f\otimes 
\phi$ is a morphism in $\calM_\Phi^n(S)$. Then there exists a totally real unit 
$r\in R^\times$ such that 
$$ f: (M,h) \rar (N,r\cdot k)\ \ \text{and}\ \ \phi: (A,\iota,\lambda) \rar (B,\jmath,r^{-1} \cdot \mu)$$
are morphisms in $\Herm_n(R)$ and $\calM_\Phi^1(S)$, respectively.
\end{lemma}
\begin{proof}
Since $f\otimes \phi$ is a morphism in $\calM_\Phi^n(S)$, we have 
$$h\otimes \lambda = (f\otimes \phi)^\vee \circ (k \otimes \mu) \circ (f\otimes \phi) = (f^\vee \otimes \phi^\vee ) \circ (k\otimes \mu) \circ (f \otimes \phi) = (f^\vee \circ k \circ f) \otimes (\phi^\vee \circ \mu \circ \phi).$$
As $\lambda$ is principal, it generates the $R$-module $\Hom_R(A,A^\vee)$, so for some $r\in R$ we have
\begin{equation} \label{tenpb} \phi^\vee \circ \mu \circ \phi = r\cdot \lambda,\ \ \ h = r\cdot (f^\vee \circ k \circ f).\end{equation}

Now we have
$$M^\vee \supseteq \im(f^\vee \circ  k \circ f) \supseteq r \cdot \im(f^\vee\circ k \circ f) = \im(h),$$
where $\im(\cdot)$ denotes the image of a map. But since $h$ as an isomorphism is surjective, we have $\im(h)=M^\vee$, therefore
 $$M^\vee = \im(h) = r \cdot \im(f^\vee \circ k \circ f) = rM^\vee.$$
By Nakayama's lemma, $M^\vee=rM^\vee$ implies $r$ is a unit in $R$. Now $\phi: A \rar B$ is an isogeny, so $\phi^\vee \circ \mu \circ \phi = r\cdot \lambda =  \lambda \circ \iota(r)$ is a polarization on $A$. By Corollary \ref{ampcor}, $r$ must be totally positive. 

On the other hand, since $r$ is a unit and $\lambda$ is an isomorphism, so is $\lambda \circ \iota(r) = \phi^\vee \circ \mu \circ \phi$. In particular, the isogeny $\phi$ is injective, and is therefore an isomorphism. The map $R \rar \Hom_R(A,B),\ r \mapsto \phi\circ \iota(r)$ is then an isomorphism of $R$-modules, and by Proposition \ref{homprop}(c), so is
$$ \Hom_R(M,N) \rar \Hom_R(M\otimes_R A,N\otimes_R B),\ \ \ g \mapsto g\otimes \phi.$$
Since $f$ maps to an isomorphism $f\otimes \phi$ on the right hand side, it must itself be an isomorphism of modules. Finally, as $r\in R$ is a totally positive unit, $(N,r \cdot k)$ and $(B,\jmath,r ^{-1}\cdot \mu)$ are objects of $\Herm_1(R)$ and $\calM_\Phi^1(S)$. The relation $(\ref{tenpb})$ now shows that $f$ and $\phi$ are structure-preserving, hence morphisms in $\Herm_n(R)$ and $\calM_\Phi^1(S)$ as claimed.
\end{proof}

We will show that a given morphism can always be written as a composition of a 
pure tensor with a morphism of a particular type, which we now construct.

Let $(\fraka,\alpha) \in \Herm_1(R)$, $(B,\jmath,\mu)\in \calM_\Phi^1(S)$, 
and suppose $(\fraka',\alpha') \in 
\Herm_1(R)$ is an inverse of $(\fraka,\alpha)$, meaning there exists an 
isomorphism
\begin{align}\label{kappa}\kappa: (\fraka,\alpha)\otimes_R (\fraka',\alpha') 
\isomto (R,\mathbbm{1}).\end{align}
Let $(N',k') = (N,k)\otimes_R (\fraka,\alpha)$ and $(B',\lambda',\mu') 
=(\fraka',\alpha')\otimes (B,\lambda,\mu)$. Then we have an isomorphism 
\begin{equation}\label{omega_k} \omega_\kappa: (N',k')\otimes (B',\jmath',\mu') 
\rar (N,k) \otimes (B,\jmath,\mu),\end{equation}
given on $T$-valued points by
$$ (\omega_\kappa)_T: (N\otimes_R \fraka)\otimes_R (\fraka' \otimes_R B(T)) 
\rar N\otimes_R B(T),\ \ \ (n\otimes x)\otimes (y \otimes t) \mapsto n\otimes 
\kappa(x\otimes y)\cdot t.$$

The map $\omega_\kappa$ is evidently an $R$-linear isomorphism of abelian 
schemes. In fact, it is also an isomorphism of triples. Verifying this amounts 
to showing the dual of an associator map is the expected associator of 
the duals. This can be checked using the explicit formula in Proposition 
\ref{duality}.

The morphism $\omega_\kappa$ above is in general not a pure tensor. Indeed, 
assume $\omega_\kappa = f\otimes \phi$. By Lemma \ref{pureten}, $\phi: 
B'=\fraka'\otimes_R B \rar B$ must be an $R$-linear isomorphism, which implies 
$\fraka' \simeq R$, hence also $\fraka \simeq R$. Then if $\fraka$ is not a 
principal ideal $\omega_k$ is not a pure tensor.

The following theorem gives an explicit description of all the morphisms 
between objects produced by Serre's construction in $\calM_\Phi^n$. 

\begin{thm}\label{morthm}Let $(A,\iota,\lambda)$, $(B,\jmath,\mu) \in 
\calM_\Phi^1(S)$ and $(M,h)$, $(N,k)\in \Herm_n(R)$, and let
$$\Psi: (M,h) \otimes (A,\iota,\lambda) \isomto (N,k) \otimes (B,\jmath,\mu)$$
be a morphism in $\calM_\Phi^n$. Then $\Psi = \omega \circ (f\otimes 
\phi)$ for some $\omega$ as in (\ref{omega_k}), and 
$$ f: (M,h) \isomto (N, k) \otimes (\fraka, \alpha), \ \ \ \phi: (A,\iota,\lambda) \isomto (\fraka' \otimes_R \alpha') \otimes (B,\jmath,\mu), $$
morphisms in $\Herm_n(R)$ and $\calM_\Phi^1$ respectively. Here, 
$\fraka = \Hom_R(A,B)$, $\fraka' = \Hom_R(B,A)$, and $\alpha$, $\alpha'$ are 
the hermitian forms in Proposition \ref{herm}(c), up to a totally positive unit 
$r\in R^\times$. 
The map $\omega$ is given 
on $T$-valued points by
$$ \omega_T: (N\otimes_R \fraka) \otimes (\fraka' \otimes_R B(T)) \isomto 
N\otimes_R B(T),\ \ \ (n \otimes f)\otimes (g\otimes t)\mapsto n \otimes 
(g\circ f \circ t),\ \ \ T\in \Sch_{/S}.$$
\end{thm}

\begin{proof} The map $\omega$ corresponds to $\omega_k$ in (\ref{omega_k}) 
where $\kappa$ is the 
isomorphism in Proposition \ref{herm}(b). We put
$$(B',\jmath',\mu')=(\fraka',\alpha')\otimes(B,\jmath, \mu),\ \ \ 
(N',k')=(N,k)\otimes(\fraka,\alpha),$$  
and $\Psi_0 = \omega^{-1} \circ \Psi$, so that
$$\Psi_0: (M,h) \otimes (A,\iota,\lambda) \isomto (N',k') \otimes (B',\jmath',\mu').$$

By Proposition \ref{herm}(a) we have $B'\simeq A$. Then $\Hom_R(A,B')\simeq R$, 
so by Proposition \ref{homprop}(c) every element of $\Hom_R(M\otimes_R A, 
N'\otimes_R B')$ is a pure tensor, including $\Psi_0$. Hence by Lemma 
\ref{pureten} there exists a totally positive unit $r\in R^\times$ and 
isomorphisms 
$$f: (M,h) \isomto (N',r\cdot k'),\ \ \ \phi: (A,\iota,\lambda) \isomto (B',\jmath',r^{-1}\cdot \mu'),$$ 
in $\Herm_1(R)$ and $\calM_\Phi^1(S)$, such that $\Psi_0 = f \otimes \phi$. In 
particular, $\Psi = \omega \circ (f\otimes \phi)$.
\end{proof}

We note that the map $\omega$ from the theorem has the same form as 
$\omega_{a,X,Y}$ of (\ref{formomeg}) from $\S2$. Indeed, with notation as in 
the theorem, put $X=(N,k)$, $Y=(B,\jmath,\mu)$, and 
$a=(\fraka,\alpha)$. The pair $(\fraka',\alpha')$ can be taken as $a^{-1}$, 
with $I_a: a \Box a^{-1} \isomto e$ given by $\kappa: 
(f\otimes g)\mapsto \iota^{-1}(g\circ f)$. Then $\omega_{a,X,Y}: 
Xa\boxtimes a^{-1}Y \isomto X \boxtimes Y$ of (\ref{formomeg}) is mapped to 
$\omega$ of Theorem \ref{morthm} by the Serre construction 
functor 
$$\Sigma_S: \Herm_n(R) \otimes_{\Herm_1(R)} \calM_\Phi^1(S) \lra \calM_\Phi^n(S),$$
$$ (M,h) \boxtimes (A,\iota,\lambda) \mapsto (M,h) \otimes (A,\iota,\lambda).$$

We now prove Proposition \ref{fullfaith}, claimed at the start of this 
section, that $\Sigma_S$ is fully faithful.

\begin{proof}[Proof of Proposition \ref{fullfaith}:] Surjectivity of $\Sigma_S$ 
on morphisms follows from Theorem \ref{morthm}, since any morphism $\Psi$ in 
$\calM_\Phi^n(S)$ has the form $\omega \circ (f\otimes
\phi)$, and $\omega$ is in the image of $\Sigma_S$ as noted above. We show 
injectivity by comparing Theorem \ref{morthm} with Proposition \ref{morprop}.

Given a morphism 
$$\tau: (M,h)\boxtimes (A,\iota,\lambda) \rar (N,k)\boxtimes (B,\jmath,\mu)$$ 
in the domain of $\Sigma_S$, by Proposition \ref{morprop} we can write $\tau = 
\omega_0 \circ (f\boxtimes \phi)$, where 
$\omega_0=\omega_{(N,k),(B,\jmath,\mu),(\fraka,\alpha)}$ depends on 
$(\fraka,\alpha)$, $(\fraka',\alpha')$ $\in \Herm_1(R)$, and an isomorphism 
$I_{(\fraka,\alpha)}: (\fraka,\alpha)\otimes_R (\fraka', \alpha') \rar 
(R,\mathbbm{1})$. 
The map $f\boxtimes \phi$ is of the form
$$ f\boxtimes \phi: (M,h) \boxtimes (A,\iota,\lambda) \rar (N',k') \boxtimes (B',\jmath',\mu'),$$
with $(B',\jmath',\mu') = (\fraka',\alpha') \otimes (B,\jmath,\mu)$ and 
$(N',k')=(N,k)\otimes_R (\fraka,\alpha)$. We can replace any of the objects 
$A$, $B$, $\fraka$ and $\fraka'$ by isomorphic ones since that does not alter 
the general form $\omega_0\circ (f\boxtimes \phi)$, as long as we also supply 
$I_{(\fraka,\alpha)}$. 

Since $\phi: A \rar B'$ is an isomorphism, we have $\fraka \simeq 
\Hom_R(A,B)$ and $\fraka'\simeq \Hom_R(B,A)$, equipped with some $\alpha, 
\alpha'$. The hermitian forms on $\fraka$, $\fraka'$ given by Proposition 
\ref{herm}(c) are $r\alpha$ and $r^{-1}\alpha'$ for some totally positive $r\in 
R^\times$. Regardless of the value of $r$, the map in Proposition \ref{herm}(b) 
can be taken as $I_{(\fraka,\alpha)}$. With this choice of $\fraka$, $\fraka'$ 
and 
$I_{(\fraka,\alpha)}$, we have $\Sigma_S(\omega_0)=\omega$, where $\omega$ is 
as in Theorem 
\ref{morthm}. 

Now let $\tau_0$ be another morphism such that 
$\Sigma_S(\tau_0)=\Sigma_S(\tau)$. We must show $\tau_0 = \tau$. By the same 
argument as above, $\tau_0$ can be written as $\omega_0 \circ (f_0 \boxtimes 
\phi_0)$ with the same $\omega_0$. Then  
$$f \otimes \phi 
= \omega^{-1} \circ \Sigma_S(\tau) = \omega^{-1}\circ \Sigma_S(\tau_0) = f_0 
\otimes \phi_0.$$ It remains to show this implies $f \boxtimes \phi = 
f_0\boxtimes 
\phi_0$.

Since $\phi_0: A \rar B'$ and $\phi: A \rar B'$ are $R$-linear isomorphisms, we 
have $\phi_0 = r \cdot \phi$ for some $r\in R$. From $f\otimes \phi = f_0 
\otimes \phi_0$, we get $f = r\cdot f_0$. Now $\iota(r) = \phi_0\circ 
\phi^{-1}$ is an automorphism of 
$(A,\iota,\lambda)$, so	 $\lambda = \iota(r)^\vee \circ 
\lambda \circ \iota(r) = \lambda \circ \iota(r^\sigma r)$, and $r^\sigma r = 
1$. It follows that $\mu_{r}:(R,\mathbbm{1}) \rar (R,\mathbbm{1}),\ x \mapsto 
rx$ is an automorphism of $(R,\mathbbm{1})$. Now, we have a diagram
$$ \begin{xymatrix} {
((M,h)\otimes_R(R,\mathbbm{1}))\boxtimes 
(A,\iota,\lambda) \ar[rr] \ar[ddd]_{(f_0\otimes \mu_r)\boxtimes \phi} \ar[dr] & 
&
(M,h)\boxtimes ((R,\mathbbm{1})\otimes (A,\iota,\lambda)\ar[ddd]^{f_0\boxtimes 
(\mu_r\otimes \phi)} \ar[dl]\\
&(M,h)\boxtimes (A,\iota,\lambda)  \ar@{..>}[d]\\
& (N',k')\boxtimes (B',\jmath',\mu') \\
((N',k')\otimes_R (R,\mathbbm{1}))\boxtimes (B',\jmath',\mu') \ar[rr] \ar[ur] 
& &(N',k')\boxtimes ((R,\mathbbm{1})\otimes (B',\jmath',\mu')), \ar[ul]}
\end{xymatrix}$$
where the oblique arrows are canonical isomorphisms. The square and the two 
triangles commute by the axioms of the categorical 
tensor product. The right trapezoid commutes if and only if the dotted arrow 
is $f_0\boxtimes (r \cdot \phi) = f_0\boxtimes \phi_0$, and the left trapezoid 
commutes if and only if it is $(r\cdot f_0)\boxtimes \phi = 
f\boxtimes \phi$. The commutativity of each trapezoid implies the other, 
therefore $f\boxtimes \phi = f_0\boxtimes \phi_0$, and $\tau = \tau_0$.
\end{proof}

\subsection{Serre construction of $\calM_\Phi^n$: the objects}

In this section we look for triples in $\calM_\Phi^n(S)$ that come from the 
Serre construction. When $S=\Spec \CC$, every object is Serre constructible 
using an equivalence between $\calM_\Phi^n(\CC)$ and a category of 
linear-algebraic data. For general $S$, every object of $\calM_\Phi^n(S)$ is 
\textit{\'etale locally on the base} Serre constructible. This is done by using 
Theorem \ref{etprop} to reduce to the complex case.

The following lemma simplifies the task of detecting Serre constructible 
triples.

\begin{lemma}\label{pairenough}
Suppose $(B,\jmath, \mu)$ and $(A,\iota,\lambda)$ are objects in 
$\calM_\Phi^n(S)$ and $\calM_\Phi^1(S)$ respectively, $M$ is a projective 
finitely presented $\OK$-module of rank $n$, and $\Psi: M\otimes_\OK A \rar B$ 
is an $\OK$-linear isomorphism of abelian schemes. Then there exists a unique 
$h: M \rar M^\vee$ such that $(M,h)\in \Herm_n(\OK)$, and $\Psi: (M,h) \otimes 
(A,\iota,\lambda) \rar (B,\jmath,\mu)$ is an isomorphism of triples.
\end{lemma}

\begin{proof}
The map $\lambda_M=\Psi^\vee \circ \mu \circ \Psi$ is an $\OK$-linear principal 
polarization on $M\otimes_\OK A$. As $\lambda$ is a basis for 
$\Hom_\OK(A,A^\vee)$, by Proposition \ref{homprop}(c) we have $\lambda_M = 
h\otimes \lambda$ for a unique $\OK$-linear map $h: M \rar M^\vee$. By 
Proposition \ref{mainprop}, $(M,h)\in \Herm_n(\OK)$. Then $(M,h)\otimes 
(A,\iota,\lambda)$ is an object of 
$\calM_\Phi^n(S)$, and $\Psi$ is an isomorphism of triples.
\end{proof}

We first consider the case $S=\Spec \CC$. For any field embedding $\phi: K \rar 
\CC$, let $\CC^{(\phi)}$ denote $\CC$ as a $K\otimes\CC$-algebra, with 
structure homomorphism $K\otimes \CC \rar \CC^{(\phi)}, a\otimes z \rar 
\phi(a)z$. Then $\CC^\Phi = \bigoplus_{\phi\in 
\Phi}\CC^{(\phi)}$ is also a $K\otimes \CC$-algebra. If $V$ is any 
$K$-vector 
space, there's a $K$-linear isomorphism
\begin{equation}\label{Cphi} V \otimes \RR \cong V \otimes_K (K \otimes \RR) 
\simeq V \otimes_K \CC^\Phi \cong \bigoplus_{\phi \in \Phi} V \otimes_{K} 
\CC^{(\phi)}.\end{equation}

\begin{lemma}\label{Cstruct} Let $(A,\iota,\lambda)\in \calM_\Phi^n(\CC)$ and 
$V=H_1(A,\QQ)$. If $V\otimes \RR$ is equipped with the $\CC$-vector space 
structure induced by $A$ as a complex variety, the isomorphism of (\ref{Cphi}) 
is $K\otimes \CC$-linear.
\end{lemma}
\begin{proof} From the Hodge filtration $\Fil^1 H_1(A,\CC) \subset H_1(A,\CC)$, 
we have
	$$ H_1(A,\RR)\cong H_1(A,\CC)/\Fil^1 H_1(A,\CC) \cong \Lie(A).$$
By Proposition \ref{Dfree}, $\Fil^1 H_1(A,\CC) = J_\Phi H_1(A,\CC)$. Then we 
have 
$K\otimes \CC$-linear isomorphisms
$$ H_1(A,\CC)/J_\Phi H_1(A,\CC) \cong 
H_1(A,\CC)\otimes_{K\otimes \CC} (K\otimes \CC / J_\Phi( K\otimes \CC) ) \cong 
H_1(A,\CC)\otimes_{K\otimes \CC} \CC^\Phi.$$
The resulting $K\otimes\CC$-linear map $H_1(A,\RR) \rar 
H_1(A,\CC)\otimes_{K\otimes\CC} \CC^\Phi$ is evidently the same as (\ref{Cphi}).
\end{proof}

By the general theory of complex abelian varieties, the isomorphism class of a 
triple $(A,\iota,\lambda)$ over $\CC$ is uniquely determined by the 
$\OK$-module $H=H_1(A,\ZZ)$, the Riemann form $E: H\times 
H \rar \ZZ$, and the complex structure on $H\otimes \RR \cong\Lie(A)$. For 
$(A,\iota,\lambda)\in \calM_\Phi^n(\CC)$, the lemma shows the complex structure 
is determined by the other data.

Let $H$ be a projective finitely presented $\OK$-module. Suppose $E:H\times H 
\rar \ZZ$ is alternating, with
\begin{align}\label{Elinear} E(\alpha 
x,y) = E(x,\alpha^\sigma y), \ \ \alpha \in \OK, x,y\in H.
\end{align} Then $E=\Tr_{K/\QQ} F$ for 
a unique $\delta_K^{-1}$-valued skew-hermitian form $F: H\times 
H \rar \delta_K^{-1}$, where $\delta_K^{-1}$ is the inverse different of 
$K/\QQ$. Let $\{\alpha_j\}$ be a $\ZZ$-basis for $\OK$, and 
$\{\beta_j\}$ the trace-dual basis for $\delta_K^{-1}$. Then 
\begin{align}\label{bigf} F(x,y) = \sum_j E(x,\alpha_j y)\beta_j.\end{align}

For any $x\in H$, $\zeta= F(x,x)$ satisfies $\zeta = -\zeta^\sigma$. Since $K$ 
is a CM field, then also $\phi(\zeta) = -\phi(\zeta)^\sigma$ for 
all $\phi\in \Hom(K,\CC)$. Hence if $\zeta\neq 0$, there exists a 
unique CM-type $\Psi \subset \Hom(K,\CC)$ such that $\Im(\psi(\zeta))< 0$ for 
all $\psi\in \Psi$. 

\begin{defn}\label{skewdef}  Suppose $\frakd\subset K$ is a fractional ideal 
satisfying $\frakd^\sigma = \frakd$. A $\frakd$-valued skew hermitian form $F: 
H\times H \rar \frakd$ is called \textit{negative-definite along} $\Phi$, if 
$\Im(\phi(F(x,x)))<0$ for all non-zero $x\in H$, $\phi\in \Phi$.
\end{defn}

Now put $U=H\otimes \RR$. Then $U \simeq H_\QQ\otimes_K 
\CC^\Phi$ via (\ref{Cphi}) is a $\CC$-vector space.

\begin{lemma}\label{negdef} $E$ is a Riemann form for $U/H$ if and only if $F$ 
is negative-definite along $\Phi$. 
\end{lemma}
\begin{proof}
Let $E_\RR$ denote the $\RR$-linear extension of $E$ to $U\times U$. Let 
$U^{(\phi)}$ be the subspace of $U=H\otimes \RR$ corresponding to 
$H_\QQ \otimes_K \CC^{(\phi)} \subset H_\QQ \otimes_K \CC^\Phi$ via 
(\ref{Cphi}). We have an orthogonal decomposition $E_\RR = \bigoplus_{\phi\in 
\Phi} E^{(\phi)}$, where $E^{(\phi)}= E_\RR|_{U^{(\phi)}\times 
U^{(\phi)}}$. Let $F_\RR= \bigoplus_{\phi\in\Phi} F^{(\phi)}$ be the 
analogous $\RR$-linear extension of $F$. Each $F^{(\phi)}$ is the 
$\CC$-linear extension of $F$ to $U^{(\phi)}\times 
U^{(\phi)}$, induced by $\phi$, with values in $\CC^{(\phi)}$.

For each $\phi\in \Phi$, using (\ref{bigf}) we have
\begin{flalign*} 
 F^{(\phi)}(x,y) = \left( \sum_j \Re(\phi(\alpha_j))\phi(\beta_j)\right) 
 E^{(\phi)}(x,y) +  \left(\sum_j \Im(\phi(\alpha_j)) \phi(\beta_j)\right) 
 E^{(\phi)}(x,iy),
 \end{flalign*}
where $\{\alpha_j\}$, $\{\beta_j\}$ are trace-dual bases for $\OK$ and 
$\delta_K^{-1}$, respectively. As $K$ is CM, we have
$$ \sum_{j} \alpha_j \beta_j =1, \ \ \ \sum_j \alpha_j^\sigma  \beta_j = 0,$$
from which 
$$ \sum_{j} \Re(\phi(\alpha_j)) \phi(\beta_j) = \frac{1}{2},\ \ \ \sum_{j} \Im(\phi(\alpha_j)) \phi(\beta_j) = \frac{i}{2}.$$
It follows that
\begin{equation} F^{(\phi)}(x,y) = \frac{1}{2} E^{(\phi)}(x,y) + \frac{i}{2} 
E^{(\phi)}(x,iy),\ \ \ \forall x,y\in U^{(\phi)},\end{equation}
and in particular $F^{(\phi)}(x,x) = -\frac{i}{2}E^{(\phi)}(ix,x)$. Then by the 
orthogonal decompositions of $F$ and $E$,  $F$ is negative-definite along 
$\Phi$ if and only if $E(ix,y)$ is positive-definite.
\end{proof}

Now put $H^*=\Hom_\OK(H,\delta_K^{-1})\cong H^\vee \otimes_\Ok 
\delta_K^{-1}$. A 
$\delta_K^{-1}$-valued sesquilinear form $F$ corresponds to an $\OK$-linear map 
$f: H \rar H^*$ via $f(x)(y)=F(x,y)$. We identify $f$ with $F$ in this 
way, and speak of skew-hermitian forms $f: H \rar H^*$. We have $(H^*)^*\cong 
H$ canonically, so that $H\mapsto H^*$ is a duality on the category of 
projective finitely 
presented $\OK$-modules. Then $f$ is skew-hermitian if and only if $f^* = -f$. 
As usual, $f$ is \textit{non-degenerate} if it is an isomorphism, and that's 
the case if and only if $E=\Tr_{K/\QQ} F$ is non-degenerate.

More generally, suppose $\frakd$ is a fractional ideal of $K$ such that 
$\frakd^\sigma = \frakd$. For $\epsilon=\pm 1$, a $\frakd$-valued sesquilinear 
form $G: H\times H \rar \frakd$ is called \textit{$\epsilon$-hermitian} if 
$H(x,y)^\sigma = \epsilon H(y,x)$. Such $G$ correspond to $\OK$-linear maps $g: 
H \rar \Hom_\OK(H,\frakd)\cong H^\vee \otimes_\OK 
\frakd$. If $(H,g)$ is $\frakd$-valued $\epsilon$-hermitian, and 
$(H',g')$ is $\frakd'$-valued $ \epsilon'$-hermitian, then $(H,g)\otimes 
(H',g')=(H\otimes_\OK H',g \otimes g')$ is $\frakd\frakd'$-valued 
$\epsilon\epsilon'$-hermitian. In particular,  $(H\otimes H',g\otimes g')$ 
is skew-hermitian if one of $(H,g)$ and $(H',g')$ is hermitian, and the other 
skew-hermitian. If $(H,g)$ and $(H',g')$ are either both skew-hermitian or 
both hermitian, $(H\otimes H',g\otimes g')$ is hermitian. If $(H,g)$ is 
$\frakd$-valued skew-hermitian and negative-definite along $\Phi$ (Definition 
\ref{skewdef}), and $(H',g')$ is $\frakd^{-1}$-valued skew-hermitian and 
negative-definite along $\Phi\sigma$, then $(H,g)\otimes (H',g')$ is 
$\OK$-hermitian and positive-definite.

Let $\SK_\Phi^n(\OK)$ denote the category of pairs $(H,f)$, where $H$ is a 
projective finitely presented $\OK$-module of rank $n$, and $f: H \rar H^*$ is 
non-degenerate $\delta_K^{-1}$-valued skew-hermitian, and negative-definite 
along 
$\Phi$. The morphisms are isomorphisms of $\OK$-modules 
which preserve the forms. The 2-group $\Herm_1(\OK)$ acts on $\SK_\Phi^n(\OK)$ 
via ordinary tensor product. In other words, if $(H,f)\in 
\SK_\Phi^n(\OK)$ and $(\fraka,\alpha)\in \Herm_1(\OK)$, then $(H\otimes_\OK 
\fraka, f\otimes \alpha)\in \SK_\Phi^n(\OK)$.

For each $(A,\iota,\lambda)\in \calM_\Phi^n(\CC)$, Let 
$\Theta(A,\iota,\lambda)=(H,f)$, where $H=H_1(A,\ZZ)$ and $f$ is
$$ H_1(A,\ZZ) \overset{H_1(\lambda)}{\lra} H_1(A^\vee,\ZZ) \cong 
H_1(A,\ZZ)^*.$$ Then we have a functor
\begin{align} \Theta=\Theta_n: \calM_\Phi^n(\CC) \lra 
\SK_\Phi^n(\OK).\end{align}

\begin{prop}\label{skewprop} $\Theta$ is an equivalence of 
categories. It is $\Herm_1(\OK)$-equivariant up to canonical 
isomorphism.
\end{prop}
\begin{proof} 
This follows from the basic theory of complex abelian varieties. A 
quasi-inverse to $\Theta$ is 
given by 
$$\Xi: \SK_{\Phi}^n(\OK) \lra \calM_\Phi^n(\CC),\ \ \ 
\Xi(H,f)=(H\otimes\RR)/H,$$
where $(H\otimes \RR)/H$ is a complex torus via (\ref{Cphi}), polarized by
$E(x,y)=\Tr_{K/\QQ}(f(x)(y))$. The isomorphism of functors 
$\mathbbm{1}_{\calM_\Phi^n(\CC)} \cong \Xi\circ \Theta$ corresponds to the canonical 
uniformization of a complex abelian variety. The other direction
$\mathbbm{1}_{\SK_\Phi^n(\OK)}\simeq \Theta\circ \Xi$ is induced by $H\cong 
H_1(H\otimes\RR/H)$. It's straight-forward to check these are isomorphisms.

Checking $\Xi$ is equivariant implies the same for $\Theta$. For $A= (H\otimes \RR)/H$, that corresponds to $\fraka 
\otimes_\OK A \cong (H_\fraka \otimes \RR) / H_\fraka$ with $H_\fraka = \fraka 
\otimes_\OK H$. That the isomorphism preserves polarizations can be 
checked using the Appell-Humbert data of the corresponding Poincar\'e bundles. 
Ultimately one must check that canonical duality isomorphisms are 
compatible with  associators. This follows from the explicit form of the duality isomorphism in Theorem \ref{duality}.
\end{proof}

For the domain of the functor $$\Sigma_\CC:\Herm_n(\OK) \otimes_{\Herm_1(\OK)} 
\calM_\Phi^1(\CC) \rar \calM_\Phi^n(\CC)$$ 
to be non-empty, $\calM_\Phi^1(\CC)$ must 
contain objects. The following theorem shows this is almost always the case.

\begin{thm}\label{existence} $\calM_\Phi^1(\CC)\neq \emptyset$ for all CM types 
$\Phi$, unless $K/F$ is unramified at every finite place. In that case, 
$\calM_\Phi^1(\CC)\neq 
\emptyset$ for exactly half the types.
\end{thm}
\begin{proof}
Assume that for some CM 
type $\Phi_0$, 
$\calM_{\Phi_0}^1(\CC)$ is non-empty. As explained in \cite[IV]{Shimura1998}
that means we can find a fractional ideal $\fraka\subset K$ and an element 
$\zeta \in K$ 
such that
$$\fraka \fraka^\sigma \delta_K = (\zeta),\ \ \zeta^\sigma=-\zeta,\ \ \Im 
\phi(\zeta)<0\ \ \forall \phi\in \Phi_0,$$
where $\delta_K$ is the different ideal of $K$. 
If $\calM_\Phi^1(\CC)\neq \emptyset$ for some other CM type $\Phi$, we get 
another pair $(\frakb, \xi)$ with the corresponding
properties. If we put $\frakc = \frakb \fraka^{-1}$ and $r = \xi \zeta^{-1}$, 
then $r\in F$, $\frakc \frakc^\sigma = (r)$, and for 
any $\psi: F \hra \Qbar$, we have $\psi(r)>0$ if and only if $\Phi \cap \Phi_0$ 
contains an element of $\Hom(K,\Qbar)$ extending $\psi$.

Let $N_0(K)$ denote the group of pairs $(\frakc,r)$, 
where $\frakc$ is a non-zero fractional ideal of $K$, $r\in F^\times$, and 
$\frakc \frakc^\sigma = (r)$. Given $(\frakc,r)\in N_0(K)$, we may put $\frakb 
= \fraka 
\frakc$ and $\xi = r\zeta $. Then $\frakb 
\frakb^{\sigma} \delta_K = (\xi)$, $\xi^\sigma= -\xi$, and there exists a 
unique 
CM 
type $\Phi$ such that 
$(\frakb,\xi)$ defines an object of $\calM_\Phi^1(\CC)$. Then $(\frakc,r)\cdot 
\Phi_0 = \Phi$ defines a transitive action of $N_0(K)$ on 
the 
set of CM types $\Phi$ for which $\calM_\Phi^1(\CC)$ is non-empty. Let 
$N_0^+(K)$ denote the kernel of this action. It coincides with the stabilizer 
of any $\Phi_0$, and consists of all pairs 
$(\frakc,r)$ such that $\frakc \frakc^{\sigma} = (r)$ and $r\in F^\times$ is 
totally positive. Then the number of CM types $\Phi$ such that $\calM_\Phi^1(\CC)\neq 
\emptyset$ is equal to $$|N_0(K)/N_0^+(K)|.$$

Let $U_K$, $I_K$, and $P_K$ denote the units of $\OK$, the non-zero 
fractional 
ideals of 
$K$, and its subgroup of principal ideals, so that $C_K = 
I_K/P_K$ is the ideal class group. We also use the corresponding notation for 
$F$. 
Let 
$N_K\subset I_K$ consist of $\frakc$ such that $\Nm_{K/F}(\frakc)=\frakc 
\frakc^\sigma$ is principal and generated by an element of $F$. We have a 
surjective map $N_0(K) \rar N_K, 
(\frakc,r)\mapsto \frakc$ 
and an exact sequence
$$ 0 \rar U_F \rar N_0(K) \rar N_K \rar 0,$$
where $u\in U_F$ is identified with $(\OK,u)\in N_0(K)$.

Let $P_F^+\subset P_F$ denote the subgroup of principal 
ideals that admit a totally positive generator, so that $C^+_F = I_F/P^+_F$ is 
the 
narrow class group of $F$. We also have a subgroup $N^+_K\subset I_K$ 
consisting of $\frakc$ such 
that $\Nm_{K/F}^+(\frakc)=\frakc\frakc^{\sigma}$ is in $P_F^+$. We get 
another exact sequence
$$ 0 \rar U_F^+ \rar N_0^+(K) \rar N_K^+ \rar 0,$$
where $U_F^+\subset U_F$ consists of totally positive units.

From the two exact sequences and the nine lemma we obtain another exact sequence
\begin{align}\label{split} 0 \rar U_F/U_F^+ \rar N_0(K)/N_0^+(K) \rar N_K/N_K^+ 
\rar 0.\end{align}

Note that $N_K$ and $N_K^+$ both contain $P_K$, so that $N_K/N_K^+ = \ol{N}_K/ 
\ol{N}^+_K$, with $\ol{N}_K = N_K/P_K$ and 
$\ol{N}^+_K 
= N_K^+/P_K$ considered as subgroups of $C_K$. We then have an exact diagram 
$$ \xymatrix{ & & 0\ar[d] & 0 \ar[d] \\
& 0 \ar[r] & \ol{N}^+_K \ar[r] \ar[d] & \ol{N}_K \ar[d] \ar[r] & 
\ol{N}_K/\ol{N}^+_K \ar[r] & 0\\
&	0 \ar[r] & C_K \ar@{=}[r] \ar[d]^{\Nm_{K/F}^+} & C_K \ar[d]^{\Nm_{K/F}} 
\ar[r] & 
0\\
0 \ar[r] & Y \ar[r] & C^+_F \ar[r] & C_F \ar[r] \ar[d] & 0,\\
& &  & 0 &
}$$
where $Y$ is kernel of $C_F^+ \rar C_F$. To see that the norm map $\Nm_{K/F}: 
C_K \rar 
C_F$ is surjective, note that under the reciprocity isomorphism it corresponds 
to the restriction 
map 
$\Gal(H_K/K)\rar \Gal(H_F/F)$. Here $H_K$ and $H_F$ are the 
Hilbert class fields of $K$ and $F$, and $H_F\subset H_K$. 

For the narrow 
Hilbert 
class fields $H_K^+$ and $H_F^+$, we have $H_K^+= H_K$ since $K$ is totally 
imaginary, and so $H_F^+\subset H_K$. Again the map $\Nm_{K/F}^+: C_K \rar 
C_F^+$ corresponds to the 
restriction $\Gal(H_K/K)\rar \Gal(H_F^+/F)$. The latter is surjective if $K 
\cap H_F^+ = F$, otherwise the image has index 2, with the quotient isomorphic 
to $\Gal(K/F)$. By a standard diagram chase we get an injection 
$\ol{N}_K/\ol{N}_K^+ \rar 
Y$, induced by $\Nm^+_{K/F}$, which is an isomorphism if $K/F$ is ramified at any 
finite prime, and an injection onto a subgroup of index 2 otherwise. 

Now fix an ordering $(\psi_1, \cdots, \psi_g)$ of $\Hom(F,\Qbar)$, put $S= 
\{\pm 
1\}^{[F:\QQ]}$, and consider the homomorphism $U_F \rar S$, $u\mapsto 
s=(s_1,\cdots, 
s_g)$, where $s_i = \psi_i(u)/|\psi_i(u)|.$ The kernel of this map is $U_F^+$, 
hence its image is isomorphic to $U_F/U_F^+$. The cokernel of this map is in 
fact isomorphic to $Y$, the kernel of $C_F^+ \rar C_F$, and also to 
$\Gal(H_F^+/H_F)$ \cite[p. 47, Lemma 
11.2]{ConHur88}. Indeed, let $F^+$ denote the totally positive elements of 
$F^\times$. By the (weak) approximation theorem the map $F^\times \rar S$ 
extending $U_F \rar S$ is surjective, so that $S \cong F^\times/F^+$. Then the 
cokernel of $U_F \rar S$ is isomorphic to $(F^\times/F^+)/(U_F/U_F^+)\cong 
F^\times/U_F F^+ \cong P_F/P_F^+ \cong Y$. It follows that $|Y|\cdot 
|U_F/U_F^+|=|S|$, and so from 
$\ol{N}_K/\ol{N}^+_K = N_K/N^+_K$ and 
(\ref{split}) we obtain 

$$|N_0(K)/N_0^+(K)| = \left\{ \begin{array}{ll} 2^{g-1} & \text{if $K/F$ is 
unramified at all finite primes} \\ 2^g & \text{ otherwise}\end{array}\right.$$

Now note the surjectivity of $\Nm_{K/F}: C_K \rar C_F$ implies our initial 
assumption that $\calM_{\Phi_0}^1(\CC)\neq \emptyset$ for some $\Phi_0$. 
Indeed, for an arbitrary totally imaginary $\zeta_0\in K^\times$, the 
fractional ideal $\zeta_0\delta^{-1}_K$ descends to a fractional ideal of $F$. 
Then there exists 
$\fraka \in I_K$ such that $N_{K/F}(\fraka) \zeta_0^{-1} \delta_K$ is 
principal, generated by some $r\in F$. In other words $\fraka 
\fraka^\sigma \delta_K = (\zeta)$, where $\zeta=r\zeta_0$ is totally imaginary. 
Letting $\Phi_0$ be the unique CM type such that $\Im(\phi(\zeta))<0$ for all 
$\phi\in \Phi_0$, we obtain a pair $(\fraka,\zeta)$ corresponding to an object 
of $\calM_{\Phi_0}^1(\CC)$.

Thus we have shown that if $K/F$ is ramified at any finite prime, the number of CM types $\Phi$ such that $\calM_\Phi^1(\CC)\neq \emptyset$ is $2^g$, which is the number of all CM types. If $K/F$ is unramified at every finite prime, then $\calM_\Phi^1(\CC)\neq \emptyset$ for $2^{g-1}$ CM types, which 
is exactly half of them.
\end{proof}

We note that if $\calM_\Phi^1(\CC)\neq \emptyset$ for some 
$\Phi$, there exists a number field $L_0$ such that $\calM_\Phi^1(\Spec 
\calO_{L_0})\neq \emptyset$. This follows the fact that CM abelian varieties 
have potential good reduction everywhere, using N\'eron models.

If $n$ is odd, we offer an explicit way to construct an object of 
$\calM_\Phi^1(\CC)$ given one in $\calM_\Phi^n(\CC)$, via Proposition 
\ref{skewprop} and the following.

\begin{prop}\label{nodd}Assume that $n=2m+1$, and $(H,f) \in \SK_\Phi^n(\OK)$. 
Then 
$(\fraka, \alpha)\in \SK_\Phi^1(\OK)$, where
	$$ \fraka = \det(H)\otimes_\OK \delta_K^m,\ \ \ \alpha = (-1)^m 
	(\det(f)\otimes \sigma).$$
\end{prop}
\begin{proof}

Since $H$ has rank $n$, $\det(H) 
= \bigwedge_\OK^n H$ is projective of rank one. Since 
$f$ is an isomorphism, so is $\det(f): \det(H) \rar \det(H^*)$. 

Letting $\fraka 
= \det(H)\otimes_\OK \delta_K^m$, we have canonical isomorphisms
$$\det(H^*) \otimes_\OK \delta_K^m \cong \det(H^\vee \otimes_\OK \delta_K^{-1}) \otimes_\OK \delta_K^m \cong \det(H)^\vee \otimes_\OK \delta_K^{m-n} \cong (\det(H)\otimes_\OK \delta_K^m)^\vee \otimes_\OK \delta_K^{-1}\cong \fraka^*,$$
where we have used $\det(H^\vee \otimes_\OK \delta_K^{-1})\cong \det(H^\vee)\otimes_\OK \delta_K^{-n}$ and $\delta_K^\vee \cong \delta_K^{-1}$ in the middle, along with $m-n=-m-1$. 

Using the above isomorphism, we obtain $\alpha: \fraka 
\rar \fraka^*$, corresponding to
$$(-1)^m (\det(f)\otimes \sigma): \det(H)\otimes_\OK \delta_K^m \rar 
\det(H^*)\otimes_\OK \delta_K^m.$$ 

Note that $\sigma: \delta_K^m \rar \delta_K^m$ corresponds to a non-degenerate positive-definite $\delta_K^{n-1}$-valued hermitian form on $\delta_K^m$ given by $(x,y)\mapsto x^\sigma y$. It follows that $\alpha$ is skew-hermitian and negative-definite along $\Phi$, if and only if $(-1)^m\det(f)$ is so. This we can verify after tensoring with $\QQ$. 

Choosing an orthogonal basis $\{h_i\}_{i=1}^n$ of $H_\QQ$ with respect to 
$F_\QQ$, $\det(H_\QQ)$ is spanned by $\eta=h_1\wedge \cdots \wedge h_n$. It 
follows 
that $\det(f_\QQ)(\eta)(\eta) = \zeta = \prod_i \zeta_i$, where $\zeta_i = 
F_\QQ(h_i,h_i)$. Since $(H,f)$ is skew-hermitian each $\zeta_i$ is purely 
imaginary. Since $n$ 
is odd so is $\zeta$, which shows $\det(f)_\QQ$ 
is skew-hermitian. As $(H,f)$ is negative-definite along $\Phi$, we have 
$\Im(\phi(\zeta_i))<0$ for all $\phi\in \Phi$. Then $\Im(\phi(\zeta)) = (-1)^m 
\prod_{i=1}^n \Im(\phi(\zeta_i))$, which shows 
$(-1)^m \det(f)_\QQ$ is negative-definite along $\Phi$. Since $\det(f): \det(H) 
\rar \det(H^*)$ and $\sigma: \delta_K^{m} \rar \delta_K^m$ are non-degenerate, 
with values in $\delta_K^{-n}$ and $\delta_K^{n-1}$ respectively, $\alpha$ is 
non-degenerate and $\delta_K^{-1}$-valued. 
\end{proof}

The following proposition is a special case of the main theorem in the next 
section.

\begin{prop}\label{basek}
Suppose $\calM_\Phi^1(\CC)\neq \emptyset$ and $k$ is an algebraically closed 
field of characteristic zero. The 
functor 
$$\Sigma_{k}: \Herm_n(\OK) \otimes \calM_\Phi^1(k) \lra \calM_\Phi^n(k)$$
induced by Serre's construction is an equivalence of categories.
\end{prop}
\begin{proof}
We can again assume $k=\CC$ by the standard descent argument. By Proposition 
\ref{fullfaith}, the functor $\Sigma_\CC$ is fully faithful, so we must show 
essential surjectivity.

We have a diagram 
$$\begin{xymatrix}{\Herm_n(\OK)\otimes_{\Herm_1(\OK)} \calM_\Phi^1(\CC) \ar[r]^-{\Sigma_\CC}\ar[d]_{\mathbbm{1} \otimes \Theta_1} & \calM_\Phi^n(\CC) \ar[d]^{\Theta_n}\\
\Herm_n(\OK)\otimes_{\Herm_1(\OK)} \SK_\Phi^1(\OK) \ar[r]_-\otimes &\SK_\Phi^n(\OK)}
\end{xymatrix}$$
which is commutative up to isomorphism. The vertical arrows are equivalences of 
categories by Proposition \ref{skewprop}. To show the top arrow is essentially 
surjective, it's enough to show the same for the bottom arrow. In other words, 
that every $(H,f)\in \SK_\Phi^n(\OK)$ is isomorphic to $(h\otimes \alpha, 
M\otimes_\OK \fraka)$ for some $(M,h)\in \Herm_n(\OK)$ and $(\fraka,\alpha)\in 
\SK_\Phi^1(\OK)$.

Let $(\fraka,\alpha)\in \SK_\Phi^1(\OK)$, which is possible since 
$\calM_\Phi^1(\CC)\neq \emptyset$. We assume $\fraka$ is a fractional ideal, 
identify $\fraka^*$ with 
$(\fraka^\sigma)^{-1}\delta_K^{-1}$, and $\alpha$ with multiplication by some 
$\zeta\in K$. Now let $\frakb = 
\fraka^{-1}$ and define $\beta: \frakb \rar \frakb^* \cong \fraka^\sigma 
\delta_K^{-1}$ by $\beta(x) = \zeta^{-1} x$. Since $(\fraka,\alpha)$ is 
skew-hermitian and non-degenerate, so is $(\frakb,\beta)$. Since 
$(\fraka,\alpha)$ is $\delta_K^{-1}$-valued and negative-definite along $\Phi$, 
$(\frakb,\beta)$ is $\delta_K$-valued and negative-definite along $\Phi\sigma$. 
We have $(\fraka\otimes_\OK \frakb, \alpha \otimes 
\beta) \simeq (\OK,\mathbbm{1})\in \Herm_1(\OK)$. 

For $(H,f) \in \SK_\Phi^n(\OK)$, let $(M,h)=(H,f)\otimes_\OK (\frakb,\beta)$. 
Since $f$ and $\beta$ are non-degenerate and skew-hermitian, with values in 
$\delta_K^{-1}$ and $\delta_K$ respectively, $h$ is $\OK$-hermitian and 
non-degenerate. As $f$ and $\beta$ are negative definite along $\Phi$ and 
$\Phi\sigma$ respectively, $h$ is positive-definite. Thus $(M,h)\in 
\Herm_n(\OK)$. Now the bottom arrow of the diagram above takes the object 
$(M,h)\boxtimes (\fraka,\alpha)$ to $(M,h)\otimes_\OK (\fraka,\alpha)$ which is 
evidently isomorphic to $(H,f)$. This shows the bottom arrow is essentially 
surjective, which was enough to prove the proposition.
\end{proof}

Now let $k$ be a finite extension of the reflex field $L$, and $S$ a connected locally noetherian scheme over $\Spec(k)$. Suppose $(A,\iota,\lambda)\in \calM_\Phi^n(S)$. 

\begin{prop}\label{fintyp}
For any $s\in S$, there exists an \'etale neighbourhood $U$ of $s$ such that the triple $(A_U,\iota_U,\lambda_U)\in \calM_\Phi^n(U)$, obtained from $(A,\iota,\lambda)$ by base change, arises from the Serre construction. In other words $(A_U,\iota_U,\lambda_U)$ lies in the essential image of the functor
$$\Sigma_S: \Herm_n(\OK) \otimes \calM_\Phi^1(U) \lra \calM_\Phi^n(U).$$
\end{prop}

\begin{proof}
Let $\calM_k = \calM_\Phi^n\otimes k \rar \Spec(k)$ be the Deligne-Mumford stack over $\Spec(k)$ obtained by base change. Since $\calM_\Phi^n \rar \Spec \OL$ is \'etale and proper by Proposition \ref{etprop}, so is $\calM_k \rar \Spec(k)$. Let $M \thrar \calM_k$ be a surjective \'etale morphism from a scheme $M$ to $\calM_k$. Then the composition $M \rar \Spec (k)$ is an \'etale morphism of schemes, so $M$ is isomorphic to a disjoint union $\coprod_{\alpha} \Spec(k_\alpha)$, with each $k_\alpha$ a finite (separable) extension of $k$.

Let $S \rar \calM_\Phi^n$ be the morphism corresponding to the triple $(A,\iota,\lambda)\in \calM_\Phi^n(S)$. It lifts uniquely to a morphism $S \rar \calM_k$, since $S \rar \Spec(\OL)$ factors through $\Spec(k) \rar \Spec(\OL)$. Let $S'=S\times_{\calM_k} M$ and consider the morphism $S' \rar M$ lying over $S \rar \calM_k$. Since $M \rar \calM_k$ is \'etale and surjective, so is $S' \rar S$.

Let $s\in S$ be a point. Let $U$ be a connected component of $S'$ containing a point $u$ mapping to $s\in S$, so that $(U,u)$ is an \'etale neighbourhood of $s$. The base change triple $(A_U,\iota_U,\lambda_U)$ corresponds to a morphism $U \rar \calM_\Phi^n$ which factors as a composition $U \rar M \rar \calM_\Phi^n$. Since $U$ is connected and $M$ is a disjoint union of $\Spec(k_\alpha)$, the map $U \rar M$ further factors through some $\Spec(k') \hra M$, where $k'$ is a finite extension of $k$. Thus, $U \rar \calM_\Phi$ can be written as a composition $U \rar \Spec k' \rar \calM_\Phi^n$. That means there exists some triple $(A_{k'},\iota_{k'},\lambda_{k'})$ such that $(A_U,\iota_U,\lambda_U)$ is the constant triple obtained from it by base change through $U \rar \Spec(k')$. Now, by Proposition \ref{basek}, the triple $(A_{k'},\iota_{k'},\lambda_{k'})$ can be obtained by Serre's construction, after passing to a finite extension $k''$ of $k'$. By replacing the \'etale neighbourhood $(U,u)$ of $s$ with a smaller one, we can assume $k''=k'$. Then $(A_{k'},\iota_{k'},\lambda_{k'}) \simeq (M,h)\otimes (A_0,\iota_0,\lambda_0)$ for some $(M,h)\in \Herm_n(\OK)$, $(A_0,\iota_0,\lambda_0)\in \calM_\Phi^1(k')$. Therefore $(A_U,\iota_U,\lambda_U) \simeq (M,h) \otimes ({A_0}_U,{\iota_0}_U,{\lambda_0}_U)$, where $({A_0}_U,{\iota_0}_U,{\lambda_0}_U)\in \calM_\Phi^1(U)$ is the triple obtained from $(A_0,\iota_0,\lambda_0)$ by base change along $U \rar \Spec(k')$. 
\end{proof}

Proposition \ref{fintyp} says if $S$ is locally noetherian over a finite extension of $L$, each triple $(A,\iota,\lambda)\in \calM_\Phi^n(S)$ can be obtained \textit{\'etale locally on the base $S$} by Serre's construction. In the next section this is generalized to any locally noetherian $S$ over $\Spec \OL$, and interpreted in terms of stacks on the big \'etale site $(\Sch_{/\OL})_{\acute{e}t}$.

\subsection{Stackification}
We assume that $\calM_\Phi^1(\CC)$ is non-empty. In 
Theorem \ref{existence} we showed this is almost always the case. In 
particular, $\calM_\Phi^1(\Spec \calO_{L_0})$ is non-empty for some 
number field $L_0$. 

Recall some of the constructions from $\S 2$. The 2-group $\Herm_1(\OK)$ acts on $\Herm_n(\OK)$ via ordinary tensor product. It also acts fibrewise on $\calM_\Phi^1$ (Definition \ref{fibrewise}), as a category fibred in groupoids over $\Sch_{/\OL}$. Put $\calT= \Herm_n(\OK)\otimes_{\Herm_1(\OK)} \calM_\Phi^1$ in the notation of $\S 2$. We then have a functor $\calT \rar \Sch_{/\OL}$ coming from the $\calM_\Phi^1$ factor.

\begin{lemma}The category $\calT$ is fibred in groupoids over $\Sch_{/\OL}$.
\end{lemma}
\begin{proof}
We apply Proposition \ref{fiberedprop} to deduce this, for which we need to show the action of $\Herm_1(\OK)$ on $\calM_\Phi^1$ is free on objects (Definition \ref{obfree}), and that $\calT$ is left-cancellative over $\Sch_{/\OL}$ (Definition \ref{relcan}). 

For $(\fraka,\alpha)$ and $(\frakb,\beta)$ in $\Herm_1(\OK)$ and $(A,\iota,\lambda)\in \calM_\Phi^1(S)$, suppose $(\fraka,\alpha)\otimes (A,\iota,\lambda)$ is isomorphic to $(\frakb,\beta)\otimes (A,\iota,\lambda)$. We claim $(\fraka,\alpha)\simeq (\frakb,\beta)$. By considering $\fraka \otimes_\OK A \simeq \frakb \otimes_\OK A$ on $A$-valued points we get $\fraka \simeq \frakb$, so without loss we can assume $\fraka = \frakb$. An isomorphism $(\fraka,\alpha)\otimes(A,\iota,\lambda) \simeq (\fraka,\beta)\otimes (A,\iota,\lambda)$ is in particular an $\OK$-linear automorphism of the abelian scheme $\fraka\otimes_\OK A$, hence of the form $\mathbbm{1}_\fraka \otimes \iota(r)$ for some $r\in \OK^\times$. 
To be an isomorphism of triples it must satisfy $$\alpha \otimes \lambda = (\mathbbm{1}_\fraka\otimes \iota(r))^\vee \circ (\beta \otimes \lambda) \circ (\mathbbm{1}_{\fraka} \otimes \iota(r)) = (\beta \otimes \lambda) \circ (\mathbbm{1}\otimes \iota(r^\sigma r) = \beta \otimes (\lambda \circ \iota(r^\sigma r))= (r^\sigma r \cdot \beta)\otimes \lambda,$$ 
which implies $r^\sigma r \cdot \beta = \alpha$. In that case the map $\mu_r: \fraka \rar \fraka,\ x \mapsto rx$ gives an isomorphism $(\fraka,\alpha)\simeq (\fraka,\beta)$. This shows the action of $\Herm_1(\OK)$ on $\calM_\Phi^1$ is free on objects.

To show $\calT$ is left-cancellative with respect to $\Sch_{/\OL}$, consider the functor $\calT \rar \calM_\Phi^n$ over $\Sch_{/\OL}$ induced by the Serre tensor construction, which over the fibres $\calT_S$ coincides with $\Sigma_S$. We first show this functor is faithful. By Proposition \ref{fullfaith} it is fully faithful on each fibre. Now suppose $\alpha: T \rar S$ is a morphism in $\Sch_{/\OL}$, and $\xi_1$, $\xi_2$ are maps $(M,h)\boxtimes (A,\iota,\lambda) \rar (N,k)\boxtimes (B,\jmath,\mu)$ in $\calT$ lying over $\alpha$, which are mapped to the same morphism in $\Sch_{/\OL}$. Any such map factors through $\gamma = \mathbbm{1}_{(N,k)}\boxtimes p_T: (N,k)\boxtimes (B_T,\jmath_T,\mu_T) \rar (N,k)\boxtimes (B,\jmath,\mu)$, where $p_T: (B_T,\jmath_T,\mu_T) \rar (B,\jmath,\mu)$ is the base change map in $\calM_\Phi^1$. Then $\xi_1 = \gamma\circ \eta_1$, $\xi_2 = \gamma \circ \eta_2$, for morphisms $\eta_1$, $\eta_2$ lying over $\mathbbm{1}_T$. The image of $\mathbbm{1}_{(N,k)}\boxtimes p_T$ under $\calT \rar \calM_\Phi^n$ is the base change map in $\calM_\Phi^n$, so it is in particular strongly cartesian. That implies $\Sigma_{T}(\eta_1) = \Sigma_T(\eta_2)$, which implies $\eta_1 = \eta_2$, since $\Sigma_T$ is faithful. This shows $\calT \rar \calM_\Phi^n$ is also faithful. Now, since $\calM_\Phi^n$ is fibred in groupoids over $\Sch_{/\OL}$, by Lemma \ref{canlem} it is left-cancellative over $\Sch_{/\OL}$. Then $\calT$ is also left-cancellative over $\Sch_{/\OL}$, by faithfulness of $\calT \rar \calM_\Phi^n$. Thus by Proposition \ref{fiberedprop}, $\calT$ is also fibred in groupoids over $\Sch_{/\OL}$. 
\end{proof}

By Proposition \ref{fullfaith}, for each $S\in \Sch_{/\OL}$ the functor $\Sigma_S: \calT(S) \rar \calM_\Phi^n(S)$ is fully faithful. Since $\calT$ is fibred in groupoids over $\Sch_{/\OL}$, the morphism $\calT \rar \calM_\Phi^n$ induced by the Serre construction is also fully faithful.

We define $\Herm_n(\OK)\otimes \calM_\Phi^1$ to be the stack associated to the presheaf $\calT$ on the big \'etale site over $\Spec \OL$. Here we have suppressed the subscript $\Herm_1(\OK)$ from the notation to differentiate it from $\calT$. Since the functor $\calT\rar \calM_\Phi^n$ is fully faithful and $\calM_\Phi^n$ is a stack, the presheaf $\calT$ is already separated. It follows that $\Herm_n(\OK)\otimes \calM_\Phi^1$ is obtained from $\calT$ by one application of the plus construction. In other words, $(\Herm_n(\OK)\otimes \calM_\Phi^1)(S)$ consists of descent data relative to \'etale coverings of $S$, with the appropriate morphisms \cite[Ch. 3, Lemma 3.2]{LaumMorChamps}.

We thus obtain a commutative diagram of categories fibred in groupoids
\begin{align}\tag{$\triangle$}
\begin{split}
\xymatrix{
& \calT \ar[dl]_-{\mathrm{Sheafification\ \ \ }} \ar[dr]^-{\mathrm{Serre\ construction}} \\ 
\Herm_n(\OK) \otimes \calM_\Phi^1 \ar[rr]^-{\Sigma} & & \calM_\Phi^n,
}\label{bigdiag}
\end{split}
\end{align}
where $\Sigma$ is the functor induced by the universal property of sheafication. 

Our results so far can be summarized as follows: the functor $\Sigma$ in the diagram $(\Delta)$ identifies $\Herm_n(\OK)\otimes \calM_\Phi^n$ with a full subcategory of $\calM_\Phi^n$. Furthermore, when $S$ is locally noetherian over a finite extension of $L$, the induced functor $(\Herm_n(\OK)\otimes \calM_\Phi^n)(S) \rar \calM_\Phi^n(S)$ is an equivalence of categories, by Proposition \ref{fintyp}. The fact that $\calM_\Phi^n$ is \'etale over $\Spec \OL$ allows us to extend this to characteristic $p$.

\begin{prop}\label{perfcharp}
Let $S=\Spec(k)$ for $k$ a perfect field of characteristic $p$. The functor $\Sigma$ in (\ref{bigdiag}) induces an equivalence of categories on the fibre over $S$.
\end{prop}
\begin{proof}
Let $W_k$ be the ring of Witt vectors of $k$, with fraction field $W_k[p^{-1}]$. Since $\calM_\Phi^n$ is smooth, an object $(A_0,\iota_0,\lambda_0)\in \calM_\Phi^n(k)$ lifts to $(A,\iota,\lambda)\in \calM_\Phi^n(W_k)$. By Proposition \ref{basek}, after base changing to a finite extension $F$ of $W_k[p^{-1}]$, we have an isomorphism $(A_F,\iota_F,\lambda_F) \simeq (M,h) \otimes (B,\jmath,\mu)$ for some $(B,\jmath,\mu) \in \calM_\Phi^1(F)$, and $(M,h)\in \Herm_n(\OK)$. Since $B$ has CM, after possibly enlarging $F$ again it has good reduction, and its N\'eron model $\scrB$ over $\OF$ is an abelian scheme \cite{SerreTate}. 

By the N\'eron mapping property, the action of $\OK$ and the polarization $\mu$ also lift to $\scrB$, so we have a triple $(\scrB,\jmath_\scrB, \mu_{\scrB})\in \calM_\Phi^1(\OF)$, and we can form $(M,h) \otimes_\OK (\scrB,\jmath_\scrB,\mu_{\scrB}) \in \calM_\Phi^n(\OF)$. Now it is easy to verify that Serre's construction commutes with taking N\'eron models, so $M \otimes_\OK \scrB $ is the N\'eron model of $M \otimes_\OK B$. On the other hand, the base change $\scrA$ of $A$ to $\Spec \OF$ is also a N\'eron model for its generic fibre, which is again $M\otimes_\OK B$, so we have $\scrA \simeq M\otimes_\OK \scrB$ by uniqueness of the model, which implies $(\scrA,\iota_{\OF}, \lambda_{\OF})\simeq (M,h) \otimes (\scrB,\jmath_\scrB,\mu_\scrB)$ by the mapping property.

Let $k'$ denote the residue field of $\OF$, a finite extension of $k$. The isomorphism $(\scrA,\iota_{\OF}, \lambda_{\OF})\simeq (M,h) \otimes (\scrB,\jmath,\mu)$ reduces modulo the prime of $\OF$ to an isomorphism $(A',\iota',\lambda') \simeq (M,h) \otimes (B',\jmath',\mu')$ over $k'$. Since $(A',\iota',\lambda')$ is the base change of $(A_0,\iota_0,\lambda_0)$ to $k'$, we have shown that an arbitrary triple $(A_0,\iota_0,\lambda_0)\in \calM_\Phi^n(k)$ arises from Serre's construction after passing to an \'etale cover $\Spec(k') \rar \Spec(k)$. Thus on the fibre over $\Spec(k)$ the functor $\Sigma$ is essentially surjective, and being fully faithful by Proposition \ref{fullfaith}, is an equivalence of categories.
\end{proof}
We can now prove the main theorem.
\begin{thm} \label{mainiso}
If $\calM_\Phi^1(\CC)\neq \emptyset$, the functor $\Sigma: \Herm_n(\OK) \otimes 
\calM_\Phi^1 \rar \calM_\Phi^n$ is an isomorphism of stacks.
\end{thm}
\begin{proof}
We will prove this by showing that $\Sigma$ induces equivalences of categories on the stalks of the geometric points of the big \'etale site on $\Spec(\OL)$. By definition these are the geometric points $\overline{s} \rar S$, where $S$ is a scheme locally of finite type over $\Spec(\OL)$. Since $\OL$ is a Jacobson ring, it's enough to only consider geometric points $\overline{s} \rar S$ having image $s\in S$ lying over a closed point $t \in\Spec(\OL)$ \cite[II, Remark 2.17(b)]{Milne_Etale}. In that case since $S$ is locally of finite type, $s\in S$ is closed and $k(t) \subset k(s)$ is a finite extension. As $k(t)$ is a finite field, $k(s)$ is perfect, and in particular Proposition \ref{perfcharp} applies to $k=k(s)$.

Let $(A,\iota,\lambda)\in \calM_\Phi^n(S)$ be a triple corresponding to a morphism $S \rar \calM_\Phi^n$ and suppose $\overline{s} \rar S$ is a geometric point whose image $s\in S$ is closed, with $k(s)$ perfect. We want to show there exists a finite \'etale morphism $U\rar S$ through which $\overline{s} \rar S$ factors, such that the base change $(A_U,\iota_U,\lambda_U)$ of $(A,\iota,\lambda)$ to $U$ is obtained by the Serre construction.

Let $(A_s,\iota_s,\lambda_s)\in \calM_\Phi^n(k(s))$ denote the fibre of $(A,\iota,\lambda)$ over $s$, which corresponds to the composition $\Spec k(s) \rar S \rar \calM_\Phi^n$. By Proposition \ref{perfcharp} there exists a finite extension $k'$ of $k(s)$ contained in $k(\overline{s})$, such that the base change $(A_{s'},\iota_{s'},\lambda_{s'})$ of $(A_s,\iota_s,\lambda_s)$ to $s'=\Spec(k')$ is isomorphic to $(M,h)\otimes(B_0,\jmath_0,\mu_0)$ for some $(M,h)\in \Herm_n(\OK)$ and $(B_0,\jmath_0,\mu_0)\in \calM_\Phi^1(k(s'))$. 

Now, let $\calS_{(M,h)}: \calM_\Phi^1 \rar \calM_\Phi^n$ be the morphism of stacks over $\Spec (\OL)$ given on sections by $(B,\jmath,\mu)\mapsto (M,h)\otimes(B,\jmath,\mu)$. Let $S' = S \times_{\calM_\Phi^n} \calM_\Phi^1$. We have a commutative diagram,
\begin{flalign*}
\begin{xymatrix}{ \overline{s} \ar[r] 
& \Spec(k')\ar@{.>}[dr] \ar@/^/[drr]^{(B_0,\jmath_0,\mu_0)} \ar@/_/[ddr]\\ & & S\times_{\calM_\Phi^n} \calM_\Phi^1 \ar[r]\ar[d] & \calM_\Phi^1 \ar^{\calS_{(M,h)}}[d] \\ & & S \ar[r]_{(A,\iota,\lambda)} & \calM_\Phi^n,
}\end{xymatrix}
\end{flalign*}
where $(A,\iota,\lambda)$ and $(B_0,\jmath_0,\mu_0)$ label their corresponding morphisms into $\calM_\Phi^n$ and $\calM_\Phi^1$, respectively.

By the universal property of the fibre product, for any morphism $U \rar S$, the base change triple $(A_U,\iota_U,\lambda_U)$ arises from Serre's construction with $(M,h)$, if and only if the map $U\rar S$ factors through $S' \rar S$. Thus $\Spec(k') \rar S$ factors as in the diagram, since $(A_{s'},\iota_{s'},\lambda_{s'})\simeq (M,h)\otimes (B_0,\jmath_0,\mu_0)$. We wish to find a finite \'etale morphism of schemes $U\rar S$, which factors through $S' \rar S$, and through which $\overline{s}\rar S$ factors. The former condition means that $(A_U,\iota_U,\lambda_U)$ arises from Serre's construction with $(M,h)$, and the latter that $U \rar S$ is an \'etale neighbourhood of $\overline{s} \rar S$. We claim that $S' \rar S$ itself is an \'etale morphism of schemes, so that we can take $U=S'$. 

We first note that $\calS_{(M,h)}: \calM_\Phi^1 \rar \calM_\Phi^n$ is \'etale, proper, and representable by algebraic spaces. Indeed, by Theorem \ref{etprop} the stacks $\calM_\Phi^n$ and $\calM_\Phi^1$ are \'etale and proper over $\Spec(\OL)$, so any $\OL$-morphism $\calM_\Phi^n \rar \calM_\Phi^1$ is \'etale and proper. The morphism $\calS_{(M,h)}$ is also representable by algebraic spaces, because tensoring with $(M,h)$ is a faithful functor $\calM_\Phi^1(T) \rar \calM_\Phi^n(T)$ for any $T\in (\Sch_{/\OL})$. It follows that $S' \rar S$ is an \'etale and proper morphism of algebraic spaces. Now any separated and locally quasi-finite morphism of algebraic spaces, in particular $S' \rar S$, is representable by schemes. Then since $S$ itself is a scheme, so is $S'$.

Then if $U=S'$, and $u$ is the composition $\overline{s} \rar \Spec(k') \rar S'$, we have the desired \'etale neighbourhood $(u,U)$ of $\overline{s} \rar S$. In other words, we have shown that for every triple $(A,\iota,\lambda)\in \calM_\Phi^n(S)$, and every geometric point $\overline{s} \rar S$ whose image $s\in S$ is closed, there exists an \'etale neighbourhood of $\overline{s} \rar S$ such that the triple $(A_U,\iota_U,\lambda_U)$ obtained by base change is in the essential image of $\Sigma_U$. Hence $\Sigma$ induces an essentially surjective functor $\Sigma_{\overline{s}/S}$ on the \'etale stalks at $\overline{s}\rar S$. Since $\Sigma$ is fully faithful by Proposition \ref{fullfaith}, $\Sigma_{\overline{s}/S}$ is an equivalence of categories. As the points $\overline{s}\rar S$ form a very dense subset of all geometric points of the big \'etale site over $\Spec \OL$, it follows that $\Sigma$ is an isomorphism of \'etale sheaves, hence an isomorphism of stacks.
\end{proof}

To prevent language from obscuring content, we restate the results in plainer terms below.

\begin{thm}\label{vanilla} Let $S$ be a connected locally noetherian scheme over $\Spec \OL$.
\begin{enumerate} 
\item[(a)] For every object $(A,\iota,\lambda)$ of $\calM_\Phi^n(S)$, and every point $s\in S$, there exists an \'etale neighbourhood $U \rar S$ of $s$, as well as objects $(M,h)\in \Herm_n(\OK)$, $(A_0,\iota_0,\lambda_0)\in \calM_\Phi^1(U)$, such that there is an isomorphism of triples
$$(A_{U}, \iota_{U},\lambda_{U}) \isomto (M,h)\otimes_{\OK} (A_0,\iota_0,\lambda_0).$$

\item[(b)] For a morphism $\phi: (A,\iota,\lambda) \rar (B,\jmath,\mu)$ of $\calM_\Phi^n(S)$, let $\{U_i \rar S\}_{i\in I}$ be a cover of $S$ by \'etale morphisms such that, as in (a), there are isomorphisms
$$\psi_{i}: (A_{U_i},\iota_{U_i},\lambda_{U_i}) \isomto (M_i,h_i)\otimes (A_i,\iota_i,\lambda_i),\ \ \ \psi_i': (B_{U_i},\jmath_{U_i},\mu_{U_i})\isomto (N_i,k_i)\otimes (B_i,\jmath_i,\lambda_i).$$
Then there exist $(\fraka_i,\alpha_i)\in \Herm_1(\OK)$, where $\fraka_i = \Hom_\OK(A_i,B_i)$, and isomorphisms
$$f_i: (M_i,h_i) \isomto (N_i,k_i) \otimes_\OK (\fraka_i,\alpha_i),\ \ \phi_i: (A_i,\iota_i,\lambda_i) \isomto (\fraka_i^{\vee},\alpha_i^\vee) \otimes (B_i,\jmath_i,\mu_i),$$
such that $\psi_i^{-1} \circ \phi \circ \psi_i = \omega_i \circ (f_i \otimes \phi_i)$ for each $i\in I$, and $\omega_i$ is a canonical isomorphism
$$ ((N_i,k_i)\otimes_\OK (\fraka_i,\alpha_i)) \otimes ((\fraka_i^\vee,\alpha_i^\vee)\otimes (A_i,\iota_i,\lambda_i)) \isomto (N_i,k_i)\otimes (A_i,\iota_i,\lambda_i).$$
\end{enumerate}
\end{thm}

\subsection{A simple example}

Let $(K,\Phi)$ consist of a quadratic imaginary extension $K/\QQ$ and an 
embedding $K \subset \CC$. Let $H$ be the Hilbert class field of $K$, and $E$ 
an elliptic curve with CM by $\OK$, defined over $H$. Let $F$ be a quadratic 
extension of $H$, and put $A = \Res^F_H E_F$. Then $A_F \simeq E_F \times E_F 
\simeq \calO_K^2\otimes_\OK E_F$. We claim $A$ itself is not isomorphic to any 
$M\otimes_\OK E'$ over $H$.

Assume $A\simeq M\otimes_\OK E'$. Since $E'$ also has CM by $\OK$, $E' \simeq 
E^\tau$ for some $\tau 
\in \Gal(H/K)$. Then $E' \simeq \fraka \otimes_\OK E$ for some fractional ideal 
$\fraka$ by the main theorem of 
complex multiplication. Hence $A\simeq M_\fraka \otimes_\OK E$ with $M_\fraka = 
M\otimes_\OK \fraka$, so $A_F \simeq 
M_\fraka \otimes_\OK E_F$. On the other hand $A_F \simeq 
\calO_K^2 \otimes_\OK E_F$, so $M_\fraka \simeq \calO_K^2$ and $A\simeq E^2$. 
This is a contradiction, since $\Res^F_H E_F$ is not isomorphic to $E\times E$.

Now fix $\jmath$ and the canonical $\mu$ such that 
$(E,\jmath,\mu)\in \calM_\Phi^1(H)$. The product polarization 
$\mu_F\times \mu_F$ commutes with the automorphism of $E_F\times E_F$ that 
switches the factors, so it descends to a polarization $\lambda$ on 
$A$. Likewise, 
the $\OK$-action on $E_F\times E_F$ descends to an $\iota$ on 
$A$, and we have $(A,\iota,\lambda)\in \calM_\Phi^2(H)$. We saw that 
$(A,\iota,\lambda)$ is not given by the Serre construction. However, the 
base change $(A_F,\iota_{F},\lambda_{F})$ is isomorphic to 
$(\calO_K^2,h)\otimes (E_F,\jmath_F,\mu_F)$, where $h: \calO_K^2 \rar 
(\calO_K^2)^\vee$ 
corresponds to $H((x_1,x_2),(y_1,y_2)) = x_1^\sigma y_1 + x_2^\sigma y_2$. 

Then $(A,\iota,\lambda)\in \calM_\Phi^2(H)$ arises from the Serre 
construction only after passing to the \'etale cover $\Spec F \rar \Spec H$. 
Equivalently, $A$ is obtained by gluing two copies of $\calO_K^2 \otimes_\OK E$ 
non-trivially along the self-intersection of the \'etale neighbourhood $\Spec F 
\rar \Spec H$. The corresponding automorphism is given 
by $(x,y)\mapsto (y,x)$ on the $\calO_K^2$ factor. Since that is not induced by 
an 
automorphism of $E$, the same gluing can not happen in 
$\calM_\Phi^1(H)$. Thus sheafifying $S \mapsto 
\Herm_n(\OK) \otimes_\OK \calM_\Phi^1(S)$ accounts for objects glued 
together by automorphisms coming from $\Herm_n(\OK)$.

\section*{Final Remarks}

The results of $\S1$, in particular Theorem \ref{mainthm} and Proposition 
\ref{mainprop}, are more general than the use we've made of them. They 
may have further applications, for instance to moduli spaces of abelian schemes 
with action by an order in a quaternion algebra.

Since Serre's construction commutes with base change, Proposition \ref{basek} allows a description of the action of $\Aut(\CC)$ on $\calM_\Phi^n$, by relating it to the description of the action on $\calM_\Phi^1$ given by the theory of complex multiplication \cite[appx A]{cmlift}. We hope to further explore this in another article.
\bibliographystyle{plain}
{\small \bibliography{refsdb}}
\end{document}